%% file: main.tex
\documentclass[a4paper, twoside, 11pt]{article}
\usepackage{a4wide, amsmath, amssymb, mathtools, yfonts}
\usepackage{float}
\usepackage{tikz}
\usepackage[all]{xy}
\usepackage[utf8]{inputenc}
\usepackage{amsthm}
\usepackage[english]{babel}
\usepackage[pdfencoding=auto,hidelinks,backref]{hyperref}
\usepackage[short-journals,initials,msc-links,backrefs]{amsrefs}
\usepackage{authblk}
\usepackage{ross_macros}
\usepackage{array}
\newcolumntype{R}[1]{>{\raggedleft\let\newline\\\arraybackslash\hspace{0pt}}m{#1}}

\numberwithin{equation}{section}
\newtheorem{lemma}{Lemma}[section]
\newtheorem{theorem}[lemma]{Theorem}
\newtheorem{proposition}[lemma]{Proposition}
\newtheorem{corollary}[lemma]{Corollary}

\newtheorem{conjecture}[lemma]{Conjecture}
\newtheorem{remark}[lemma]{Remark}
\newtheorem{example}[lemma]{Example}
\theoremstyle{definition}
\newtheorem{definition}[lemma]{Definition}

\newtheorem{assumption}[lemma]{Assumption}

\newcommand{\eps}{\epsilon}
\newcommand{\bm}[1]{\boldsymbol{#1}}
\renewcommand{\l}{\left}
\renewcommand{\r}{\right}
\renewcommand{\ss}{\substack}
\newcommand{\f}{\frac}
\newcommand{\op}[1]{\operatorname{#1}}
\newcommand{\mc}[1]{\mathcal{#1}}
\newcommand{\mf}[1]{\mathfrak{#1}}
\newcommand{\bw}{\bm{w}}

\newcommand{\bv}{\bm{v}}
\newcommand{\bu}{\bm{u}}
\newcommand{\bz}{\bm{z}}
\newcommand{\bb}{\bm{b}}
\newcommand{\ba}{\bm{a}}

\newcommand{\br}{\bm{r}}
\newcommand{\bt}{\bm{t}}

\newcommand{\N}{\mathbb{N}}
\newcommand{\Z}{\mathbb{Z}}
\newcommand{\Q}{\mathbb{Q}}
\newcommand{\R}{\mathbb{R}}
\newcommand{\C}{\mathbb{C}}
\newcommand{\A}{\mathbb{A}}
\newcommand{\FF}{\mathbb{F}}
\newcommand\Proj{\mathbb{P}}

\newcommand\Gal{\mathrm{Gal}}
\newcommand\Norm{\mathrm{N}}

\newcommand\Div{\mathrm{Div}(K)}

\newcommand{\RBw}{\mathcal{R}_{\bm{B}, \bw}}

\newcommand{\mn}[2][n]{\left[#2\right]_{#1}}

\bbsymbols{Z,Q,R,P,A,b,C}
\boldsymbols{v,w,A,a,B,s,x,b,c,t,e}
\calsymbols{C,B,A,I,J,M,X}
\upsymbols{Aut,ord,Hom}
\fraksymbols{p,m}

\DeclareMathOperator{\Br}{Br}

\title{\vspace{-\baselineskip}\sffamily\bfseries Local solubility of generalised Fermat equations}
\author[1]{Peter Koymans\thanks{Mathematisch Instituut, Universiteit Utrecht, Postbus 80.010, 3508 TA Utrecht, The Netherlands, p.h.koymans@uu.nl}}
\author[2]{Ross Paterson\thanks{School of Mathematics, School of Mathematics, University of Bristol, Bristol, BS8 1TW, UK, and the Heilbronn Institute for Mathematical Research, Bristol, UK, rosspatersonmath@gmail.com}}
\author[3]{Tim Santens\thanks{Centre for Mathematical Sciences, Wilberforce Rd, Cambridge CB3 0WB, ts996@cam.ac.uk}}
\author[2]{Alec Shute\thanks{School of Mathematics, University of Bristol, Bristol, BS8 1TW, UK, and the Heilbronn Institute for Mathematical Research, Bristol, UK, alec.shute@bristol.ac.uk}}
\affil[1]{Utrecht University}
\affil[2]{University of Bristol}
\affil[3]{University of Cambridge}
\date{\today}

\begin{document}
\maketitle

\begin{abstract}
For every $n\geq 2$ we determine the asymptotic formula for the number of integer triples $(a,b,c)$ of bounded absolute value such that the generalised Fermat equation given by $ax^n+by^n+cz^n=0$ is everywhere locally soluble.  We compute the leading constant, answering a question of Loughran--Rome--Sofos, and determine that the conjectures of Loughran--Smeets and Loughran--Rome--Sofos hold for such equations.
\end{abstract}

\input{Section1}
\input{Section2}
\input{Section3}
\input{Section4}
\input{Section5}
\input{Section6}

\bibliography{bibliography}
\end{document}

%% file: Section1.tex
\section{Introduction}\label{sec:introduction}
Following Wiles's \cite{Wiles} proof of Fermat's last theorem, generalised Fermat equations 
$$
ax^p + by^q + cz^r = 0
$$
have been extensively studied in the literature, see for example \cites{BCDY, Darmon, DG, Ellenberg, FNS}. These papers ultimately rest on the modular method introduced by Wiles, and results are only known in a very limited range of values for the coefficients $a,b,c$ and the exponents $p, q, r$.

In this article we study the family of generalised Fermat equations of the form
\begin{align}
\label{eGenFermat}
ax^n + by^n + cz^n = 0,
\end{align}
where $(a, b, c)$ is a triple of integers of absolute value at most $B$.  The question of how many such equations have a non-zero rational solution $(x, y, z)$ appears to be completely out of reach to present techniques, with the exception of the much larger family of all cubic forms due to Bhargava \cite{Bhargava}.  In the light of this, one instead studies how often the equation has a solution everywhere-locally (i.e.~over every $\QQ_p$ and over $\RR$).  This perspective has seen significant interest in arithmetic geometry, with landmark conjectures of Loughran--Smeets \cite{LS} and Loughran--Rome--Sofos \cite{LRS} predicting the asymptotic for the number of everywhere-locally soluble members in certain families of varieties.  Indeed, in the latter \eqref{eGenFermat} is described as an interesting special case.  These conjectures have attracted a great deal of recent interest, see for example \cites{BL, BLS, LTT, LM, SV, Wilson, Wolff}.

Specifically for \Cref{eGenFermat}, the problem was first considered in the case $n=2$ by Serre \cite{Serre}, Guo \cite{Guo}, Friedlander--Iwaniec \cite{FI},  and Loughran--Rome--Sofos \cite{LRS}.  The case of general $n$ was first considered by Browning--Dietmann \cite{BD}, who gave an asymptotic upper bound which is sharp up to a constant when $n$ is prime.  Our main result solves the problem in total generality.

\begin{theorem}
\label{tMain}
Let $n > 1$ be an integer. Then as $B\to\infty$
\[
\# \set{(a, b, c)\in\ZZ^3 : \substack{\abs{a},\abs{b},\abs{c} \leq B, \\ \gcd(a,b,c) = 1, \\ ax^n+by^n+cz^n=0\textnormal{ has a} \\\textnormal{non-zero solution everywhere-locally}}}\sim C_n B^3 (\log B)^{3\alpha_n - 3},
\]
where, writing $\phi$ for Euler's phi function,
\[
\alpha_n = \frac{1}{\phi(n)} \sum_{r \in (\Z/n\Z)^{\times}}\frac{1}{\gcd(n, r - 1)},
\]
and $C_n > 0$ is an explicit constant which is independent of $B$.
\end{theorem}

Our proof techniques are inspired by work of Heath-Brown \cite{HB} and Fouvry--Kl\"uners \cite{FK}, which have not been previously applied to the area until the work of Loughran--Rome--Sofos \cite{LRS} in the case $n=2$.  Much of the novelty in our approach lies in the use of higher order Dirichlet characters to study the situation $n>2$.  In particular, this result verifies the Loughran--Smeets conjecture \cite{LS} for these generalised Fermat equations.

An obvious corollary of \Cref{tMain} is an upper bound for the count of equations with a rational solution.

\begin{corollary}
Let $n > 1$ be an integer. Then as $B\to\infty$
\[
\# \set{(a, b, c) \in \ZZ^3 : \substack{\gcd(a, b, c) = 1, \\ \abs{a},\abs{b},\abs{c} \leq B, \\ a x^n + b y^n + c z^n = 0 \textnormal{ has a} \\ \textnormal{non-zero rational solution}}} \ll B^3(\log B)^{3\alpha_n-3},
\]
where the implied constant depends only on $n$, and $\alpha_n$ is as in \Cref{tMain}.
\end{corollary}

This corollary verifies an analogue of a conjecture of Poonen--Voloch. In \cite{PV}*{Conjecture 2.2(i)} they predict that for integers $n>m+1$, among degree $n$ homogeneous polynomials in $\ZZ[X_0,\dots,X_m]$ with coefficients bounded by $B$ the proportion which have a rational point in $\PP^m$ is $o(1)$. In this language, we are considering $m=2$ and all $n>1$ (not just $n>3$), and since our power of $\log$ is always negative we obtain the analogue of their conjecture for diagonal equations.

\subsection*{The Leading Constant}
Recently, in \cite{LRS}, Loughran--Rome--Sofos have presented a refinement to the Loughran--Smeets conjecture which presents the conjectural leading constant.  Motivated by this, we present the constant $C_n$ from \Cref{tMain} explicitly as follows.

\begin{theorem}
\label{thm:leadingconstant}
The constant $C_n>0$ in \Cref{tMain} is given by
\[
C_n = (1 + \mathbf{1}_{\textnormal{even}}(n))\frac{\delta_\infty(n)}{\Gamma(\alpha_n)^3 }\prod_{p\textnormal{ prime}}\delta_p(n)\braces{1-\frac{1}{p}}^{3\alpha_n - 3}, 
\]
where: $\Gamma, \phi$ are Euler's gamma and phi; $\mathbf{1}_{\textnormal{even}}$ is the indicator function for even integers;
\begin{align*}
    \alpha_n&:=\frac{1}{\phi(n)}\sum_{r \in (\Z/n\Z)^{\times}}\frac{1}{\gcd(n,r-1)};
    \\\delta_\infty(n)&:=\begin{cases}8&\textup{if }n\textnormal{ is odd,}\\6&\textup{if }n\textnormal{ is even;}\end{cases}\\
    \delta_p(n)&:=\textnormal{Vol}_{\ZZ_p^3}\braces{\set{(a_1,a_2,a_3)\in\ZZ_p^3~:~\substack{p\nmid a_i\textnormal{ for some }i\textnormal{ and}\\a_1x^n+a_2y^n+a_3z^n=0\textnormal{ has}\\\textnormal{a non-zero solution in }\ZZ_p}}}.
\end{align*}
\end{theorem}

In \Cref{sLRS} we show that this is equal to the conjectured constant of Loughran--Rome--Sofos \cite{LRS} for these families of Fermat equations, hence verifying the Loughran--Rome--Sofos conjecture in this case.  In proving the result, we perform the relevant Hensel's lemma calculation which describes the local densities $\delta_p$ explicitly as follows.

\begin{proposition}
\label{intro:prop:explicitdensities}
The local densities $\delta_p(n)$ in \Cref{thm:leadingconstant} are given explicitly by
\[
\delta_p(n) = \frac{1 - p^{-3n}}{(1 - p^{-n})^3} \frac{m_1(p)}{p^{3(n + 2v_p(n))}} + \frac{3(1 - p^{-2n})}{(1 - p^{-n})^3} \frac{m_2(p)}{p^{3(n + 2v_p(n))}} + \frac{3}{(1 - p^{-n})^2} \frac{m_3(p)}{p^{3(n + 2v_p(n))}}
\]
where:
\begin{align*}
    m_1(p) &:=  \# \left\{(a_1, a_2, a_3) \in (\Z_p/ n^2 p^n\Z_p)^3: 
        \begin{array}{c}
             p \nmid a_1,a_2,a_3,   \\
             a_1 t_1^n + a_2 t_2^n + a_3 t_3^n \emph{ soluble in } \Z_p/n^2p^n\Z_p
        \end{array}\right\};
    \\m_2(p) &:= \# \left\{(a_1, a_2, a_3) \in (\Z_p/ n^2 p^n\Z_p)^3: 
        \begin{array}{c}
             0 < v_p(a_1) < n, p \nmid a_2,a_3,   \\
             a_1 t_1^n + a_2 t_2^n + a_3 t_3^n \emph{ soluble in } \Z_p/n^2p^n\Z_p
        \end{array}\right\};
    \\m_3(p) &:= \# \left\{(a_1, a_2, a_3) \in (\Z_p/ n^2 p^n\Z_p)^3: 
        \begin{array}{c}
              0 < v_p(a_1) = v_p(a_2) < n, p \nmid a_3,   \\
             a_1 t_1^n + a_2 t_2^n + a_3 t_3^n \emph{ soluble in } \Z_p/n^2p^n\Z_p
        \end{array}\right\}.
\end{align*}
Moreover, if $\frac{p+1}{\sqrt{p}}>(n-1)(n-2)$ and $p \nmid n$, then $\delta_p(n) \cdot \left(1 - \frac{1}{p}\right)^{-3}$ equals
\[
\frac{1 - p^{-3n}}{(1 - p^{-n})^3} + \f{3 (1 - p^{-2n})(p^{-1}-p^{-n})}{\gcd(n,p-1) (1 - p^{-1}) (1 - p^{-n})^3} + \f{3 (p^{-2}-p^{-2n})}{\gcd(n,p-1) (1 - p^{-n})^2 (1-p^{-2})}.
\]
\end{proposition}

\noindent As an example we provide a computation of the constant in the case $n = 3$.

\begin{example}
Let us work what our theorem says for $n = 3$ explicitly. In this case $\alpha_3 = \frac{1}{2}(\frac{1}{3} +  1)  = \frac{2}{3}$ so the order of magnitude is $B^3 (\log B)^{-1}$.

We will now work out the local densities for the Euler product. Note that every prime number $p$ satisfies $\frac{p+1}{\sqrt{p}}>2$, and so by \Cref{intro:prop:explicitdensities}, we have
\[
\delta_p(3) = 
\begin{cases}
\frac{p^2-p+1}{(p-1)^2} & \textnormal{if } p \equiv 1 \bmod 3, \\
\frac{p^6+3p^5+6p^4+p^3+6p^2+3p+1}{(p-1)^2(p^2+p+1)^2} & \textnormal{if } p \equiv 2 \bmod 3.
\end{cases}
\]
It remains to determine $\delta_3(3)$.  Let us now compute the $m_i(3)$. Note that we have to check solubility modulo $3^5 = 243$. A direct computation shows that
\[
\delta_3(3) = \frac{757}{676} \cdot \frac{56}{243} + 3 \cdot \frac{189}{169} \cdot \frac{320}{6561} + 3 \cdot \frac{729}{676} \cdot \frac{80}{6561} = 0.4611 \cdots.
\]
We can then numerically compute the Euler product for the primes $p < 10000$. We find that it is equal to $1.212 \cdots$. As $n$ is odd and $\Gamma(\frac{2}{3}) = 1.354 \cdots$, we conclude that 
\[
C_3 = \frac{8}{\Gamma(\frac{2}{3})^3} \prod_p \delta_p(3) \left(1 - \frac{1}{p}\right)^{2 - 3} = 3.910 \cdots. 
\]
\end{example}

\subsection*{Organisation of paper}
\Cref{tMain} (and \Cref{thm:leadingconstant}) are presented together as \Cref{thm:main summary}, which can be seen as the culmination of the sections which precede it.  Throughout, we fix the exponent $n$ which appears in these.  In \Cref{sec:prerequisites} we present prerequisites on characterising local solubility, specifically we apply a form of Hensel's lemma to the problem to characterise the local solubility via congruence conditions.  Following this, in \Cref{sec:charactersums} we consider the counting problem of \Cref{tMain} with the additional stipulation that the $a_i$ are $n$-powerfree, and rephrase this count as a sum of products of power residue symbols (i.e. as a character sum).  

In \Cref{sEqui} we prove results analogous to the bilinear sieve and the Siegel--Walfisz theorems for our setting, and apply these to isolate the main term in the character sum.  Having done this, we are then left with the easier problem of determining the number of coprime triples $(a,b,c)$ of $n$-powerfree integers (i.e. we have accounted for the everywhere-locally soluble constraint and can remove this from the problem).  In \Cref{sLeading} we then solve this remaining problem using multiple zeta analysis, and then account for the $n$-powerfree condition to finally conclude in the form presented in \Cref{tMain}.  We conclude the article with \Cref{sLRS}, where we show that the leading term we obtain is precisely the one conjectured by Loughran--Rome--Sofos in \cite{LRS}.

\subsection*{Acknowledgements}
We are grateful to Tim Browning, Daniel Loughran and Carlo Pagano for preliminary discussions on this topic. This project was initiated and worked on during the Workshop on Arithmetic and Algebra of Rational Points (WAARP), hosted at the University of Bristol. The first author gratefully acknowledges the support of Dr. Max R\"ossler, the Walter Haefner Foundation and the ETH Z\"urich Foundation. The first author also acknowledges the support of the Dutch Research Council (NWO) through the Veni grant ``New methods in arithmetic statistics''. The second and fourth authors were supported by the University of Bristol and the Heilbronn Institute for Mathematical Research. The third author was supported by a PhD fellowship of FWO Vlaanderen.

\subsection*{Notation}
In the course of the argument, we must introduce a number of bookkeeping notations and rely on some standard notations.  For the ease of the reader, we have provided a reference for these here.  We typically use bold font notations to indicate vectors (or vectors of functions).  We fix an algebraic closure $\bar{\QQ}$ of the rational numbers.

Throughout the article, we fix an integer $n>1$, and a primitive $n$-th root of unity $\zeta_n\in\bar{\QQ}$.  For $d\mid n$ we write $\zeta_d:=\zeta_n^{n/d}$.  In particular, $\ll$ and big-$O$ notation refer to an asymptotic upper bound whose implicit constant may depend on $n$.

Further notations are presented below:

\begin{center}
\begin{tabular}{| R{3.2cm} | m{10.8cm} |} \hline
$\be_i$ & $i$-th standard basis vector, i.e. $i$-th of $(1,0,0), (0,1,0)$, or $(0,0,1)$ in $\ZZ^3$ \\ \hline
$\PP^r_F$ & $F$-points in $r$-space over $F$\\ \hline
$\ba$ & triple $\ba=(a_1,a_2,a_3)$\\ \hline
$\bb=\bb(\ba)$ & minimised triple associated to $\ba$, see \Cref{def:b(a)}\\ \hline
$\cC_{n,\ba}$ & projective curve $a_1x_1^n+a_2x_2^n+a_3x_3^n=0\subseteq \PP^2$, see \Cref{def:[a] and C_na}\\ \hline
$[a]_n$ & representative of $a\bmod n$ contained in $\set{0,\dots,n-1}$, see \Cref{def:[a] and C_na}\\ \hline
$S(n)$ & set of primes which are small relative to $n$, see \Cref{def:S(n)}\\ \hline
$\mathcal{D}(n)$ & subset of $\RR^3$ of coefficients which are $\RR$-soluble, see \Cref{lem:RealSolubility}\\ \hline
$M$& bad modulus relative to $n$, see equation \eqref{eHensel} and text beneath\\ \hline
$\mathbf{1}_P$ & indicator function that property $P$ holds\\ \hline
$N_{\text{f}}(\bB)$ & counting function of interest with $n$-powerfree constraint, see equation \eqref{eq:NfB}\\ \hline
$N_{\text{f}}(\bB;\bs)$ & as above with additional coprimality constraints from $\bs$, see equation \eqref{eq:NfBs}\\ \hline
$N_{\text{f}}(\bB;\bs,\bb,M)$ & counting function of study in \Cref{sec:charactersums,sEqui}, see equation \eqref{eq:NfBsbM}\\ \hline
$\cA$ & subset of $(\ZZ/n\ZZ)^3$, see \Cref{subsubsec:COV}\\ \hline
$\cB$ & subset of $(\ZZ/n\ZZ)^3$, see \Cref{subsubsec:COV}\\ \hline
$\cB_{\min}$ & subset of $\set{0,1}^3$, see \Cref{subsubsec:auxiliaryfunction}\\ \hline
$((x_\bv)_{\bv\in \cB}, (x_{i,\bullet})_{i=1}^3)$&change of variables from \Cref{lem:initialCOV}\\ \hline
$I_{K}$ & monoid of ideals of ring of integers of $K$\\ \hline
$f$& during \Cref{sec:charactersums,sEqui} $f$ is a bookkeeping function defined in \Cref{subsubsec:auxiliaryfunction}, during \Cref{sLeading} it denotes a completely multiplicative function as defined in \eqref{eq:completelymult}\\ \hline
$g$&multiplicative function, see \Cref{eq:mult function g}\\ \hline
$\tilde{g}$ & completely multiplicative function agreeing on primes with $g$\\ \hline
$\pi_i$ & $i$-th naive projection $(\ZZ/n\ZZ)^3\to\ZZ$, see \Cref{subsubsec:COV}\\ \hline
$\art{\cdot}{\cdot}_{\QQ(\zeta_d),d}$ & a fixed $d$-th power residue symbol\\ \hline
$F,\,G,\,H$ & multiple zeta functions, see \Cref{def:H(s)} and text beneath\\ \hline
$K$ & multiple zeta function, see \Cref{def:K(s)} \\ \hline
\end{tabular}
\end{center}

\noindent Notation in \Cref{sLRS} is assumed from the work of Loughran--Rome--Sofos \cite{LRS}.

%% file: Section2.tex
\section{Local solubility}\label{sec:prerequisites}
Since we will be interested in the local solubility of generalised Fermat equations, we will now focus on characterising this behaviour.  

\subsection{Definitions and notation}
Throughout this section, we let $p$ be a prime number. 

\begin{definition}\label{def:[a] and C_na}
For $a\in\ZZ$, we write $\mn{a}\in\set{0,\dots,n-1}$ for the representative of $a \bmod n$. For each field $F$ and triple $\ba=(a_1, a_2, a_3)\in F^3$, we define the projective curve
\[
\cC_{n,\ba}: a_1t_1^n + a_2t_2^n + a_3t_3^n = 0 \subseteq \PP^2_F.
\]
\end{definition}

In the case of $\QQ_p$, we will associate to a triple of coefficients a corresponding minimised triple.

\begin{definition}\label{def:b(a)}
    Let $\ba=(a_1,a_2,a_3)\in (\QQ_p^\times)^3$ be a triple such that 
    $$
    \#\set{[v_p(a_i)]_n~:~i \in \{1, 2, 3\}} \leq 2.  
    $$
    Let $i < j$ be entries such that $[v_p(a_i)]_n=[v_p(a_j)]_n$. Then we define the associated minimised triple to be $\bb=\bb(\ba)\in\ZZ_p^3$ whose entries are
    \begin{align*}
        b_1&:=\frac{a_i}{p^{v_p(a_i)}}\\
        b_2&:=\frac{a_j}{p^{v_p(a_j)}}\\
        b_3&:=\frac{a_k}{p^{v_p(a_k)-\mn{v_p(a_k)-v_p(a_i)}}}.
    \end{align*}
\end{definition}

\begin{lemma}
\label{lem:Ca to Cb}
Let $\ba=(a_1,a_2,a_3)\in (\QQ_p^\times)^3$ be a triple such that $\#\set{[v_p(a_i)]_n~:~i \in \{1, 2, 3\}}\leq 2$, and let $\bb$ be the associated minimised triple. Then
\[
\cC_{n,\ba}\cong_{\QQ_p} \cC_{n,\bb}.
\]
In particular, $\cC_{n,\ba}(\QQ_p)\neq \emptyset$ if and only if $\cC_{n,\bb}(\QQ_p)\neq \emptyset$.
\end{lemma}

\begin{proof}
Let $i < j$ be entries such that $[v_p(a_i)]_n = [v_p(a_j)]_n$ and let $k$ be the unique element of $\set{1, 2, 3} - \set{i, j}$.  We scale the equation by $p^{-v_p(a_i)}$, so that our triple $(a_1, a_2, a_3)$ becomes
\[
\braces{\frac{a_i}{p^{v_p(a_i)}},\ \frac{a_j}{p^{v_p(a_i)}}, \frac{a_k}{p^{v_p(a_i)}}}
\]
after potentially re-ordering. Now since the valuation of the middle entry is a multiple of $n$, we scale $t_2$ by the appropriate power to reach the triple
\[
\braces{\frac{a_i}{p^{v_p(a_i)}},\ \frac{a_j}{p^{v_p(a_j)}}, \frac{a_k}{p^{v_p(a_i)}}}.
\]
Note that the final term has valuation which is congruent to $[v_p(a_k)-v_p(a_i)]_n\bmod n$, and so scaling $t_3$ by an appropriate power of $p$ we reach the triple $\bb$.  
\end{proof}

\subsection{Hensel's Lemma}
We recall the statement of Hensel's lemma, in our specific circumstance.

\begin{lemma}[Hensel's lemma]
\label{lem:AdaptedHensel}
Let $(a_1, a_2, a_3) \in \ZZ_p^3$. Assume that there exists $(x_1, x_2, x_3) \in \ZZ_p^3$ such that
\[
v_p(a_1x_1^n + a_2x_2^n + a_3x_3^n) > 2(v_p(n) + \min_i\set{v_p(a_ix_i^{n-1})}).
\]
Then there exists a point $(t_1, t_2, t_3)\in\ZZ_p^3$ which represents a point $Q \in \cC_{n, \bf{a}}(\QQ_p)$ and such that $v_p(t_i-x_i) > v_p(n) + \min_i\set{v_p(a_ix_i^{n-1})}$.
\end{lemma}

\begin{proof}
This is a consequence of Hensel's lemma, as presented in \cite{Browning}*{Theorem 1.13}. The fact that the corresponding lift provided by that theorem is not $(t_1, t_2, t_3) = (0, 0, 0)$ is clear from the final inequality since we cannot have $v_p(x_i) > (n - 1) v_p(x_i)$.
\end{proof}

Now we give the criteria for local solubility which we make use of in our arguments in this article.

\begin{proposition}
\label{prop:Hensel_general}
Let $\ba = (a_1, a_2, a_3)\in (\QQ_p^\times)^3$.  Then
\begin{itemize}
    \item If $\#\set{[v_p(a_i)]_n~:~i \in \{1, 2, 3\}} = 3$, then $\cC_{n,\ba}(\QQ_p) = \emptyset$.
    \item If $\#\set{[v_p(a_i)]_n~:~i \in \{1, 2, 3\}} \leq 2$, then write $\bb = \bb(a) = (b_1, b_2, b_3)$ for the minimised triple associated to $\ba$. Then $\cC_{n,\ba}(\QQ_p) \neq \emptyset$ if and only if there exists $\mathbf{t} = (t_1, t_2, t_3) \in \ZZ_p^3$ such that:
\begin{enumerate}
    \item $\mathbf{t}\not\equiv (0, 0, a) \bmod p$ for all $a \in \Z/p\Z$ and \label{enum:Hensel1}
    \item $b_1t_1^n + b_2t_2^n + b_3t_3^n \equiv 0 \bmod p^{2v_p(n) + 1}$.\label{enum:Hensel2}
\end{enumerate}
\end{itemize}
\end{proposition}

\begin{proof}
We begin with the case that $\#\set{\mn{v_p(a_i)}:i\in\set{1, 2, 3}} = 3$. Assume that $(t_1, t_2, t_3)\in\ZZ_p^3$ represents a point on $\cC_{n, \bf{a}}(\QQ_p)$. Then since the three $\mn{v_p(a_i)}$ are distinct, we get
\[
v_p(a_1t_1^n + a_2t_2^n + a_3t_3^n) = \min\set{v_p(a_it_i^n)} < \infty,
\]
a contradiction.  

Henceforth we will assume that $\#\set{\mn{v_p(a_i)}:i\in\set{1, 2, 3}}\leq 2$. Then, by \Cref{lem:Ca to Cb}, $\cC_{n,\ba}(\QQ_p)\neq \emptyset$ if and only if $\cC_{n,\bb}(\QQ_p) \neq \emptyset$, and so we show that the two constraints proposed are equivalent to $\cC_{n,\bb}(\QQ_p)\neq \emptyset$. We first show that the forward implication holds.  If $\cC_{n,\bb}(\QQ_p)\neq\emptyset$, then certainly any $\QQ_p$-point can be normalised to a $\ZZ_p$-point $\bt = (t_1, t_2, t_3)$ such that $\bt \not\equiv (0, 0, 0) \bmod p$.  We must further show that $\bt \not\equiv (0, 0, a) \bmod p$ for any choice of $a\in \FF_p^\times$. Indeed, by the definition of a minimised triple, $v_p(b_1) = v_p(b_2) = 0$ and $v_p(b_3)\in\set{0,\dots,n-1}$, and so reducing modulo $p^n$ we would obtain $b_3t_3^n \equiv 0 \bmod p^n$. Hence $t_3 \not\equiv 0 \bmod p$ implies $b_3 \equiv 0 \bmod p^n$, which is a contradiction. Hence $\bt$ satisfies the constraints proposed. 

Assume now that the two constraints in the proposition statement hold, so we have $\bt \in \Z_p^{3}$ such that $\bt \not\equiv (0, 0, a) \bmod p$ and $b_1t_1^n + b_2t_2^n + b_3t_3^n \equiv 0\bmod p^{2v_p(n) + 1}$. By \Cref{lem:AdaptedHensel} and Constraint \Cref{enum:Hensel2}, if 
$$
\min\set{v_p\braces{b_it_i^{n-1}}~:~i\in\set{1, 2, 3}} = 0,
$$
then $\cC_{n,\bb}(\QQ_p)\neq\emptyset$, and so we will show that this holds. Indeed, assume to the contrary. By construction of $\bb$, we have $v_p(b_1) = v_p(b_2) = 0$ and $v_p(b_3) \in \set{0, \dots, n - 1}$. Moreover, by Constraint \Cref{enum:Hensel1}, at least $1$ entry of $\bt$ has valuation $0$.  Since we are assuming that $v_p(b_it_i^{n - 1})>0$ for all $i$, we must have that $v_p(t_3)=0$ and $v_p(t_1),v_p(t_2)>0$. But this contradicts our assumptions on $\bf{t}$. 
\end{proof}

In the generic case, the conditions above can be phrased simply.  We now give a notation for the set of badly behaved primes for whom the generic case need not apply.

\begin{definition}\label{def:S(n)}
We let $S(n)$ be the set of prime numbers given by
\[
S(n) := \set{p \mid n} \cup \set{p~:~\frac{p+1}{\sqrt{p}} \leq (n-1)(n-2)}.
\]
\end{definition}

For $p\not\in S(n)$, we can phrase the Hensel result more simply.

\begin{proposition}
\label{prop:Hensel_large_primes}
Let $p$ be a prime number, and $\ba = (a_1, a_2, a_3) \in (\QQ_p^{\times})^3$.  Assume that $p \not \in S(n)$, then $\cC_{n, \ba}(\QQ_p) \neq \emptyset$ if and only if one of the following is true:
\begin{enumerate}
\item $\#\set{\mn{v_p(a_i)}~:~i\in\set{1, 2, 3}} = 1$; or
\item $\#\set{\mn{v_p(a_i)}~:~i\in\set{1, 2, 3}} = 2$ and we have $-(a_i/a_j) p^{v_p(a_j) - v_p(a_i)} \in \FF_p^{\times n}$, where $i < j$ are the unique indices such that $\mn{v_p(a_i)} = \mn{v_p(a_j)}$.
\end{enumerate}
\end{proposition}

\begin{proof}
Let $\bb=\bb(\ba)$ be the associated minimised triple.  If $\#\set{\mn{v_p(a_i)}:i\in\set{1, 2, 3}} = 1$, then the isomorphic curve in $\PP^2_{\FF_p}$ cut out by 
\[
b_1 t_1^n + b_2 t_2^n + b_3 t_3^n = 0
\]
is a smooth genus $\frac{(n-1)(n-2)}{2}$ curve, and so has a point by the Hasse--Weil bound. If instead $\#\set{\mn{v_p(a_i)}:i\in\set{1, 2, 3}} = 2$, then Proposition \ref{prop:Hensel_general} gives that $\cC_{n, \ba}(\QQ_p)\neq\emptyset$ if and only if there is a point $(t_1, t_2) \in \FF_p^2\backslash\set{(0, 0)}$ such that $b_1t_1^n + b_2t_2^n \equiv 0\bmod p$. Since $b_1b_2\not\equiv 0\bmod p$, we rearrange to see that the existence of this point is equivalent to existence of $t_0 = (t_2/t_1)\in\FF_p^\times$ such that $t_0^n = -(b_1/b_2) = -(a_i/a_j)p^{v_p(a_j) - v_p(a_i)}$. 
\end{proof}

%% file: Section3.tex
\section{The character sum}\label{sec:charactersums}
We let $\mathbf{B} = (B_1, B_2, B_3)$. Our main sum of interest is
\begin{equation}\label{eq:NfB}
N(\mathbf{B}) := \sum_{\substack{|a_1| \leq B_1 \\ |a_2| \leq B_2 \\ |a_3| \leq B_3 \\ \gcd(a_1, a_2, a_3) = 1}} \mathbf{1}_{a_1t_1^n + a_2t_2^n + a_3t_3^n = 0 \text{ E.L.S.}}
\end{equation}
Setting $\mathbf{a} = (a_1, a_2, a_3)$, we will write the coefficient bounds compactly as $\mathbf{a} \leq \mathbf{B}$.  It will be useful to first work in the case where the entries of $\ba$ are $n$-powerfree, together with some additional coprimality constraints.  For $\mathbf{s} = (s_1, s_2, s_3)$, we define
\begin{equation}\label{eq:NfBs}
N_{\text{f}}(\mathbf{B}, \mathbf{s}) := \sum_{\substack{\mathbf{a} \leq \mathbf{B} \\ \gcd(s_1a_1, s_2a_2, s_3a_3) = 1 \\ \mathbf{a} \ n \text{-powerfree}}} \mathbf{1}_{a_1t_1^n + a_2t_2^n + a_3t_3^n = 0 \text{ E.L.S.}}
\end{equation}
Our next few sections will be aimed at understanding $N_{\text{f}}(\mathbf{B}, \mathbf{s})$.  To accommodate local solubility at primes in $S(n)$, it is helpful to be able to partition the sum into finitely many congruences modulo some fixed $M$.  This will take the form of the further generalization
\begin{equation}\label{eq:NfBsbM}
N_{\text{f}}(\mathbf{B}; \mathbf{s}, \mathbf{b}, M) := \sum_{\substack{\mathbf{a} \leq \mathbf{B} \\ \gcd(s_1a_1, s_2a_2, s_3a_3) = 1 \\ \mathbf{a} \ n \text{-powerfree} \\ \mathbf{a} \equiv \mathbf{b} \bmod M}} \mathbf{1}_{a_1t_1^n + a_2t_2^n + a_3t_3^n = 0 \text{ E.L.S.}}
\end{equation}
for $\mathbf{b} = (b_1, b_2, b_3), \mathbf{s} = (s_1, s_2, s_3)$ and $M \in \Z_{\geq 1}$, where we think of $\mathbf{b}$ and $M$ as fixed and $\mathbf{s}$ as very small compared to $\mathbf{B}$. We will see how to get back from $N_{\text{f}}(\mathbf{B}; \mathbf{s}, \mathbf{b}, M)$ to $N(\mathbf{B})$ in Section \ref{sLeading}. 

From now on, we will always assume that $M$ is divisible by
\begin{align}
\label{eHensel}
n^2 \prod_{p\in S(n)} p^n,
\end{align}
and that $M$ is not divisible by any primes outside $S(n)$. If $a_1, a_2, a_3$ are $n$-powerfree, then the local solubility of $a_1t_1^n + a_2t_2^n + a_3t_3^n=0$ at primes in $S(n)$ is determined by $a_i \bmod M$ due to Proposition \ref{prop:Hensel_general} (since we have to scale $a_i$ by at most $p^{n - 1}$ to change variables to the associated minimised triple).

We say that $\mathbf{b}\in\braces{\ZZ/M\ZZ}^3$ is admissible if there exist $n$-powerfree coprime integers $a_1, a_2, a_3$ such that $a_i \equiv b_i \bmod M$ and the equation $a_1t_1^n + a_2t_2^n + a_3t^n$ is soluble modulo $M$.  Note that, by \Cref{prop:Hensel_general}, for $n$-powerfree triples $\ba$ we have that $a_1t_1^n+a_2t_2^n+a_3t_3^n$ is locally soluble at all primes in $S(n)$ if and only if $\ba \equiv \bb \bmod M$ for some admissible triple $\bb$. For now, we shall fix $\mathbf{b}$ and shall implicitly assume that it is admissible.


\subsection{Unfolding the indicator function}
\subsubsection{Change of variables}\label{subsubsec:COV}
Define $\mathcal{A} := (\Z/n\Z)^3 - \{(0, 0, 0)\}$. Given $\mathbf{v} = (v_1, v_2, v_3) \in \mathcal{A}$ and a vector $\mathbf{a} = (a_1, a_2, a_3)$ of $n$-powerfree integers, we split $\mathbf{a}$ according to the valuations of its coordinates
\[
x_{\mathbf{v}}=x_\bv(\ba) := \prod_{\substack{p \not \in S(n) \\ v_p(a_i) \equiv v_i \bmod n}} p,
\]
where the product runs over all primes $p$. This is a finite product as $\mathbf{v} \neq \mathbf{0}$.  We also define
\[
x_{i, \bullet} = x_{i,\bullet}(\ba) := \text{sgn}(a_i) \prod_{p \in S(n)} p^{v_p(a_i)}.
\]
These new variables prescribe a change of variables $\mathbf{a}\to (x_{\mathbf{v}})_{\mathbf{v}\in \mathcal{A}}$ since, writing $\pi_i: \mathcal{A} \rightarrow \Z$ for the projection map onto the $i$-th coordinate composed with the natural section of $[\cdot]_n$, we have the identity
\[
a_i = x_{i, \bullet} \prod_{\mathbf{v} \in \mathcal{A}} x_{\mathbf{v}}^{\pi_i(\mathbf{v})},
\]
where we used that the $a_i$ are free of $n$-th powers.

The new variables $x_{\mathbf{v}}$ are squarefree and pairwise coprime. We let $\mathcal{B}$ be the subset of $\mathcal{A}$ consisting of elements $(v_1, v_2, v_3)$ such that there exists $i \in \{1, 2, 3\}$ with $v_i = 0$ and such that there exists $j, k \in \{1, 2, 3\}$ with $v_j = v_k$.  Note that the coprimality of the $a_i$ and local solubility conditions from \Cref{prop:Hensel_large_primes} force $x_{\mathbf{v}} = 1$ for all $\mathbf{v} \in \cA - \cB$ when we range over the locally soluble coprime $(a_1,a_2,a_3)$, and so we restrict our index set for $x_{\bv}$ from $\cA$ to $\cB$.

We summarise the outcome of this change of variables as follows.

\begin{lemma}
\label{lem:initialCOV}
The change of variables $\ba\mapsto \big((x_{i,\bullet})_{i=1}^3, (x_\bv)_{\bv\in\cB}\big)$ above induces a bijection
\begin{gather*}
\set{\ba\in\ZZ^3~:~\substack{
    \ba\leq \bB;\\
    \ba\ n\textnormal{-powerfree};\\
    \gcd\braces{\set{s_ia_i~:~i=1,2,3}}=1;\\
    \ba \equiv \bb \bmod M;\\
    a_1t_1^n+a_2t_2^n+a_3t_3^n=0\textnormal{ E.L.S.}
}} \\
\updownarrow \\
\set{\big((x_{i,\bullet})_i, (x_\bv)_\bv\big)\in\ZZ_{\neq 0}^3\times \ZZ_{>0}^{\cB}~:~\substack{
    \mu^2\braces{\prod\limits_{\bv\in\cB}x_{\bv}}=1;\\
    p\mid x_{\bv}\implies p\not\in S(n);\\
    p\mid x_{i,\bullet}\implies p\in S(n)\textup{ and } v_p\braces{x_{i,\bullet}}<n;\\
    \left|x_{i,\bullet}\prod\limits_{\bv\in\cB}x_{\bv}^{\pi_i(\bv)}\right| \leq B_i\ \forall i;\\
    \gcd\braces{\set{s_ix_{i,\bullet}\prod\limits_{\bv\in\cB}x_\bv^{\pi_i(\bv)}~:~i=1,2,3}}=1;\\
    x_{i,\bullet}\prod\limits_{\bv\in\cB}x_{\bv}^{\pi_i(\bv)} \equiv b_i \bmod M\ \forall i;\\
    \sum\limits_{i=1}^3\braces{x_{i,\bullet}\prod\limits_{\bv\in\cB}x_{\bv}^{\pi_i(\bv)}}t_i^n=0\textnormal{ E.L.S.}
}}.
\end{gather*}
\end{lemma}

\subsubsection{The auxiliary function \texorpdfstring{$f$}{f}}\label{subsubsec:auxiliaryfunction}
We equip $(\Z/n\Z)^3$ with an operator $\cdot$ of type $(\Z/n\Z)^3 \times (\Z/n\Z)^3 \rightarrow \Z/n\Z$ given by
\[
(v_1, v_2, v_3) \cdot (w_1, w_2, w_3) := v_1w_1 + v_2w_2 + v_3w_3.
\]
We now define a book-keeping function $f:\cB\to (\ZZ/n\ZZ)^3$. For $\bv\in\cB$, let $\set{i, j, k} = \set{1, 2, 3}$ be an identification such that $v_i=v_j$.  Then choose the image $f(\bv) := (w_1, w_2, w_3)$ such that $w_i=1$, $w_j=-1$, $w_k=0$.  Note that this defines each $f(\bv)$ up to a choice of sign (i.e. up to swapping $i$ and $j$).  That choice will be irrelevant in this article, but for the sake of concreteness we make the choice as follows.  Note that every element in $\mathcal{B}$ is a multiple of an element in the set
\[
\mathcal{B}_{\text{min}} := \{(1, 0, 0), (0, 1, 0), (0, 0, 1), (1, 1, 0), (1, 0, 1), (0, 1, 1)\}.
\]
In fact, given $\mathbf{v} \in \mathcal{B}$, there is a unique $\mathbf{v}_{\text{min}} \in \mathcal{B}_{\text{min}}$ such that $\mathbf{v}$ is a multiple of $\mathbf{v}_{\text{min}}$. We define $f: \mathcal{B} \rightarrow (\Z/n\Z)^3$ such that $f(\mathbf{v})=f(\mathbf{v}_{\min})$ and the values on $\cB_{\min}$ are stated below.
\begin{table}[H]
\begin{center}
\begin{tabular}{|l|l|}
\hline
$\mathbf{v}_{\text{min}} \in \mathcal{B}_{\text{min}}$ & $f(\mathbf{v}_{\text{min}})$ \\ \hline
$(1, 0, 0)$ & $(0, 1, -1)$ \\ \hline
$(0, 1, 0)$ & $(-1, 0, 1)$ \\ \hline
$(0, 0, 1)$ & $(1, -1, 0)$ \\ \hline
$(1, 1, 0)$ & $(1, -1, 0)$ \\ \hline
$(1, 0, 1)$ & $(-1, 0, 1)$ \\ \hline
$(0, 1, 1)$ & $(0, 1, -1)$ \\ \hline
\end{tabular}
\caption{Table of values for $f$ on $\cB_{\min}$.}
\label{table1}
\end{center}
\end{table}

\begin{remark}
Write $\Delta := (1, 1, 1) \in (\Z/n\Z)^3$ for the diagonal element.  Note that the function $f$ satisfies
\[
f(\mathbf{v}) \cdot \Delta = 0, \quad f(\mathbf{v}) \cdot \mathbf{v} = 0, \quad \pi_i(f(\mathbf{v})) \in (\Z/n\Z)^\ast \textnormal{ for some } i.
\]
\end{remark}

\subsubsection{The solubility criteria}
We now characterise the local solubility of the Fermat equations of interest.  Since we account for solubility at $p\in S(n)$ by the congruence conditions imposed by $\bb,M$ in $N_{\text{f}}(\bB;\bs,\bb,M)$, it remains to characterise solubility at $p\not\in S(n)$ and over the reals.  We begin with the former.  For a (characteristic $0$) field $K$ containing the $n$-th roots of unity, we write $(\tfrac{\cdot}{\cdot})_{K,n}$ for the $n$-th power residue symbol in $K$.  We let $g$ be the multiplicative function defined on any prime power $p^r$, $r\geq 1$ as
\begin{equation}
\label{eq:mult function g}
g(p^r) = 
\begin{cases}
1 & \text{if } p \in S(n) \\
\max(\{k \mid n : p \equiv 1 \bmod k\})^{-1} & \text{if } p \not \in S(n),
\end{cases}
\end{equation}
and define $g(-1) = 1$.  We then phrase the condition that $(a_1,a_2,a_3)$ presents a $\QQ_p$-soluble equation at $p\not\in S(n)$ in terms of this function $g$ and the change of variables of \Cref{lem:initialCOV}.

We first repackage the Hensel's lemma result from \Cref{sec:prerequisites} in a convenient form.

\begin{lemma}
\label{lem:repackaged hensel}
Let $(a_1, a_2, a_3)$ be a triple of coprime $n$-powerfree integers. We denote by $\big((x_{m,\bullet})_m, (x_\bv)_\bv\big)$ the corresponding variables under \Cref{lem:initialCOV}. Let $p$ be a prime number.  

If $p \nmid x_{\bv}$ for all $\bv\in \cB$ and $p \not \in S(n)$, then the equation $a_1t_1^n + a_2t_2^n + a_3t_3^n = 0$ is soluble in $\QQ_p$.  Else, if $p \mid x_\bv$ for some $\bv\in \cB$, then $a_1t_1^n + a_2t_2^n + a_3t_3^n = 0$ is soluble in $\Q_p$ if and only if
\[
-\prod_{m = 1}^3 x_{m, \bullet}^{\pi_m(f(\mathbf{v}))} \prod_{\mathbf{w} \in \mathcal{B}} x_{\mathbf{w}}^{\mathbf{w} \cdot f(\mathbf{v})}
\]
is an $n$-th power modulo $p$. 
\end{lemma}

\begin{proof}
The claim is immediate from \Cref{prop:Hensel_large_primes}.  
\end{proof}

\begin{lemma}
\label{lem:indicator function}
Let $(a_1,a_2,a_3)$ be a triple of coprime $n$-powerfree integers. Let $\big((x_{m,\bullet})_m, (x_\bv)_\bv\big)$ be the corresponding variables under \Cref{lem:initialCOV}.  Let $p$ be a prime number such that $p \mid x_{\mathbf{v}}$ for some $\mathbf{v} \in \mathcal{B}$.  The indicator function for $a_1t_1^n+a_2t_2^n+a_3t_3^n=0$ being soluble in $\QQ_p$ equals
\[
\mathbf{1}_{a_1t_1^n + a_2t_2^n + a_3t_3^n = 0 \textnormal{ soluble in }\QQ_p}=g(p) \sum_{d \mid n} \sum_{\substack{I \subseteq \Z[\zeta_d] \\ \Norm(I) = p}} \left(\frac{-\prod_{m = 1}^3 x_{m, \bullet}^{\pi_m(f(\mathbf{v}))} \prod_{\mathbf{w} \in \mathcal{B}} x_{\mathbf{w}}^{\mathbf{w} \cdot f(\mathbf{v})}}{I}\right)_{\Q(\zeta_d), d}.
\]
\end{lemma}

\begin{proof}
By \Cref{lem:repackaged hensel}, it is sufficient to show that the indicator function for an integer $a$ which is coprime to $p$ to be an $n$-th power mod $p$ is
\[
g(p)\sum_{d\mid n}\sum_{\substack{I\subseteq \ZZ[\zeta_d]\\\Norm(I)=p}}\art{a}{I}_{\QQ(\zeta_d),d},
\]
which we now do.  By character orthogonality, the indicator function we seek is equal to
\[
\frac{1}{\#\FF_p^\times/\FF_p^{\times n}}\sum_{\chi:\FF_p^\times/\FF_p^{\times n}\to \CC^\times}\chi(a)=g(p)\sum_{d\mid n}\sum_{\substack{\chi:\FF_p^\times\to \CC^\times\\\text{order }d}}\chi(a)
\]
where we have used that $g(p)=\frac{1}{\#\FF_p^\times/\FF_p^{\times n}}$.  We then conclude by observing the equality of the following two sets of functions on $\FF_p$
\[
\left\{\left(\frac{\cdot}{I}\right)_{\Q(\zeta_d), d} :\substack{ I \subseteq \Z[\zeta_d]\\ \Norm(I) = p}\right\}
=
\set{\chi~:~\substack{\text{Dirichlet character}\\\textnormal{modulo }p\\\textnormal{of exact order }d}}.
\]
\end{proof}

To conclude, we characterise solubility over the real numbers as follows.

\begin{lemma}
\label{lem:RealSolubility}
Let $(a_1,a_2,a_3)$ be a triple of coprime $n$-powerfree integers, and denote by $\big((x_{i,\bullet})_i, (x_\bv)_\bv\big)$ the corresponding variables under \Cref{lem:initialCOV}.  Then $a_1t_1^n+a_2t_2^n+a_3t_3^n=0$ is soluble over $\RR$ if and only if $(x_{i,\bullet})_i \in \mathcal{D}(n)$ where $\mathcal{D}(n)$ equals $\mathbb{R}^3$ if $n$ is odd and $\mathcal{D}(n)$ equals the subset $(y_1, y_2, y_3) \in \mathbb{R}^3$ with not all $y_i$ of equal sign if $n$ is even.
\end{lemma}

\begin{proof}
This is clear from the fact that the sign of $a_i$ is equal to the sign of $x_{i, \bullet}$.
\end{proof}

\subsection{Rewriting the sum}
We begin by re-expressing $N_{\text{f}}(\bB,\bs,\bb,M)$ using the characterisations of local solubility above.  Recall that we have chosen $M$ to be divisible by the expression in \Cref{eHensel}, and so since $\bb$ is chosen to be admissible then by \Cref{prop:Hensel_general} the condition that our triple $(a_1, a_2, a_3) \equiv \bb \bmod M$ ensures local solubility at primes in $S(n)$.  

We will now apply the change of variables from \Cref{lem:initialCOV}, and detect local solubility at $p\not\in S(n)$ and over the reals using \Cref{lem:indicator function} and \Cref{lem:RealSolubility}.  Hence $N_{\text{f}}(\mathbf{B}; \mathbf{s}, \mathbf{b}, M)$ equals
\[
\sum_{(x_{i, \bullet})_{1 \leq i \leq 3}} \sum_{(x_\mathbf{v})_{\mathbf{v} \in \mathcal{B}}}^\flat
\prod_{\mathbf{v} \in \mathcal{B}} \prod_{p \mid x_\mathbf{v}} \left(g(p) \sum_{d \mid n} \sum_{\substack{I \subseteq \Z[\zeta_d] \\ \Norm(I) = p}} \left(\frac{-\prod_{m = 1}^3 x_{m, \bullet}^{\pi_m(f(\mathbf{v}))} \prod_{\mathbf{w} \in \mathcal{B}} x_{\mathbf{w}}^{\mathbf{w} \cdot f(\mathbf{v})}}{I}\right)_{\Q(\zeta_d), d}\right),
\]
where $\sum^\flat$ denotes the summation conditions
\begin{gather*}
\abs{x_{i, \bullet} \prod_{\mathbf{v} \in \mathcal{B}} x_{\mathbf{v}}^{\pi_i(\mathbf{v})}} \leq B_i, \quad \gcd\braces{\set{s_i x_{i, \bullet} \prod_{\mathbf{v} \in \mathcal{B}} x_{\mathbf{v}}^{\pi_i(\mathbf{v})} : i \in \{1, 2, 3\}}} = 1, \\
p \mid x_{i, \bullet} \Rightarrow p \in S(n)\textnormal{ and }v_p(x_{i, \bullet})<n, \\
p \mid x_\mathbf{v} \Rightarrow p \not \in S(n), \quad \mu^2\left(\prod_{\mathbf{v} \in \mathcal{B}} x_\mathbf{v}\right) = 1, \\
x_{i, \bullet} \prod_{\mathbf{v} \in \mathcal{B}} x_{\mathbf{v}}^{\pi_i(\mathbf{v})} \equiv b_i \bmod M, \quad (x_{i, \bullet})_{1 \leq i \leq 3} \in \mathcal{D}(n).
\end{gather*}
We now observe that, for each fixed integer $a$, the function that maps $z$ to
\[
g(z) \sum_{\substack{(I_d)_{d \mid n} \\ I_d \subseteq \Z[\zeta_d] \\ \prod_{d \mid n} \Norm(I_d) = z}} \prod_{d \mid n} \left(\frac{a}{I_d}\right)_{\Q(\zeta_d), d}
\]
is multiplicative on squarefree coprime integers. Let $z = z_1 z_2$ with $z_1, z_2$ squarefree and coprime. Indeed, this follows from the following two observations
\begin{itemize}
\item if we let $D(z)$ be the set of Dirichlet characters modulo $z$ of order dividing $n$, then we have a bijection $D(z_1) \times D(z_2) \rightarrow D(z_1 z_2)$, given by $(\chi_1, \chi_2) \mapsto \chi_1 \chi_2$, thanks to the Chinese remainder theorem;
\item the sum over $(I_d)_{d \mid n}$ with $\prod_{d \mid n} \Norm(I_d) = z$ is exactly the sum over all Dirichlet characters modulo $z$ of order dividing $n$.
\end{itemize}
This allows us to rewrite $N_{\text{f}}(\mathbf{B}; \mathbf{s}, \mathbf{b}, M)$ as
\begin{align}
\label{eCharacterSum}
\sum_{(x_{i, \bullet})_{1 \leq i \leq 3}} \sum_{(x_\mathbf{v})_{\mathbf{v} \in \mathcal{B}}}^\flat \prod_{\mathbf{v} \in \mathcal{B}} \left(g(x_\bv) \hspace{-0.3cm} \sum_{\substack{(I_d)_{d \mid n} \\ I_d \subseteq \Z[\zeta_d] \\ \prod_{d \mid n} \Norm(I_d) = x_\mathbf{v}}} \prod_{d \mid n} \left(\frac{-\prod_{m = 1}^3 x_{m, \bullet}^{\pi_m(f(\mathbf{v}))} \prod_{\mathbf{w} \in \mathcal{B}} x_\mathbf{w}^{\mathbf{w} \cdot f(\mathbf{v})}}{I_d}\right)_{\Q(\zeta_d), d}\right).
\end{align}

We continue to manipulate equation (\ref{eCharacterSum}). We swap the product and the sum to obtain
\[
\sum_{(x_{i, \bullet})_{1 \leq i \leq 3}} \sum_{(x_\mathbf{v})_{\mathbf{v} \in \mathcal{B}}}^\flat \hspace{-0.3cm} \sum_{\substack{(I_{\mathbf{v}, d})_{\mathbf{v} \in \mathcal{B}, d \mid n} \\ I_{\mathbf{v}, d} \subseteq \Z[\zeta_d] \\ \prod_{d \mid n} \Norm(I_{\mathbf{v}, d}) = x_\mathbf{v}}} \hspace{-0.3cm} g\left(\prod_{\mathbf{v} \in \mathcal{B}} x_\mathbf{v}\right) \prod_{\mathbf{v} \in \mathcal{B}} \prod_{d \mid n} \left(\frac{-\prod_{m = 1}^3 x_{m, \bullet}^{\pi_m(f(\mathbf{v}))} \prod_{\mathbf{w} \in \mathcal{B}} x_\mathbf{w}^{\mathbf{w} \cdot f(\mathbf{v})}}{I_{\mathbf{v}, d}}\right)_{\Q(\zeta_d), d}.
\]
Since $\prod_{d \mid n} \Norm(I_{\mathbf{v}, d}) = x_\mathbf{v}$ for all $\bv$, the sum may be rewritten to allow the ideals to range freely without fixing the tuple $(x_\bv)_{\bv}$ in the outer sum.  Hence we summarise with the following theorem.

\begin{theorem}
\label{thm:Nf as character sum}
We have that
\begin{multline*}
N_{\textup{f}}(\mathbf{B}; \mathbf{s}, \mathbf{b}, M) = \sum_{(x_{i, \bullet})_{1 \leq i \leq 3}} \sum_{\substack{(I_{\mathbf{v}, d})_{\mathbf{v} \in \mathcal{B}, d \mid n} \\ I_{\mathbf{v}, d} \subseteq \Z[\zeta_d]}}^{\flat \flat} g\left(\prod_{\mathbf{v} \in \mathcal{B}} \prod_{d \mid n} \Norm(I_{\mathbf{v}, d})\right) \times \\ 
\prod_{\mathbf{v} \in \mathcal{B}} \prod_{d \mid n} \left(\frac{-\prod_{m = 1}^3 x_{m, \bullet}^{\pi_m(f(\mathbf{v}))} \prod_{\mathbf{w} \in \mathcal{B}} \prod_{e \mid n} \Norm(I_{\mathbf{w}, e})^{\mathbf{w} \cdot f(\mathbf{v})}}{I_{\mathbf{v}, d}}\right)_{\Q(\zeta_d), d},
\end{multline*}
where $\flat \flat$ imposes the summation conditions
\begin{gather*}
\left|x_{i, \bullet} \prod_{\mathbf{v} \in \mathcal{B}} \prod_{d \mid n} \Norm(I_{\mathbf{v}, d})^{\pi_i(\mathbf{v})}\right| \leq B_i, 
\quad \gcd(\{s_i x_{i, \bullet} \prod_{\mathbf{v} \in \mathcal{B}} \prod_{d \mid n} \Norm(I_{\mathbf{v}, d})^{\pi_i(\mathbf{v})} : i \in \{1, 2, 3\}\}) = 1, \\
p \mid x_{i, \bullet} \Rightarrow p \in S(n) \ \textup{ and } \ v_p(x_{i,\bullet})<n,\\
p \mid \Norm(I_{\mathbf{v}, d}) \Rightarrow p \not \in S(n), \quad \mu^2\left(\prod_{\mathbf{v} \in \mathcal{B}} \prod_{d \mid n} \Norm(I_{\mathbf{v}, d})\right) = 1, \\
x_{i, \bullet} \prod_{\mathbf{v} \in \mathcal{B}} \prod_{d \mid n} \Norm(I_{\mathbf{v}, d})^{\pi_i(\mathbf{v})} \equiv b_i \bmod M, \quad (x_{i, \bullet})_{1 \leq i \leq 3} \in \mathcal{D}(n).\\
\end{gather*}
\end{theorem}

Section \ref{sEqui} is devoted to extracting the main term from this character sum. For $n$ odd, the main term comes from terms with $I_{\mathbf{v}, d} = (1)$ for all $\mathbf{v} \in \mathcal{B}$ and all $d \neq 1$. For any other choice of $I_{\mathbf{v}, d}$, we will show that the sum oscillates. If $n$ is even, then the main term is more subtle. For now, let us just remark that one situation in which the sum oscillates is when there is at least one pair $\mathbf{v} \in \mathcal{B}$ with $d > 2$ and $I_{\mathbf{v}, d} \neq (1)$. 

%% file: Section4.tex
\section{Isolating the main term}
\label{sEqui}
This section is dedicated to extracting the main term from the summation in \Cref{thm:Nf as character sum} by character sum methods. The essential idea is as follows. We have a sum of products of Hecke characters. Freezing all but two of the variable inputs, we note that as the remaining two ideals vary, we are summing an expression of the form
\[
\art{N(J)}{I}_{\QQ(\zeta_d),d} \art{N(I)}{J}_{\QQ(\zeta_e),e}
\]
for some fixed $d,e$ and with $I,J$ varying up to some norm bound (dependent on the fixed variables) in the appropriate monoid of ideals. We seek to exploit oscillation of this function to bound the contribution in this case. Certainly, if $d=e=1$ then this is a trivial character, and so there is no oscillation to exploit. Similarly, quadratic reciprocity also causes trivial instances (if $n$ is even), and so these cases will contribute to our main term. Otherwise, we exploit a bilinear sieve estimate if $I,J$ are comparably large to get cancellation, and if $I$ is notably larger than $J$, then we exploit a form of the Siegel--Walfisz theorem to do the same. This will allow us to reduce to the trivial cases, and so the main term.

\subsection{Analytic preliminaries}
We now present the analytic tools we make use of in order to isolate the main term from \Cref{thm:Nf as character sum}.  We begin by presenting some elementary bounds, before going on to give the forms of the bilinear sieve and Siegel--Walfisz theorem that we will use.

\subsubsection{Elementary bounds}
We will make use of the following elementary bounds.

\begin{lemma}
\label{lLargeGCD}
For $B_1,B_2,Y \geq 2$, we have
\[
\sum_{\substack{|a|\leq B_1\\|b| \leq B_2 \\ \gcd(a, b) > Y}} 1 \ll \frac{B_1B_2}{Y}.
\]
\end{lemma}

\begin{proof}
Note that the lemma is trivial if $Y > \min\set{B_1,B_2}$, since the sum is empty in that case. So we may assume that $Y \leq \min\set{B_1,B_2}$. We fix the value of $\gcd(a, b)$ to bound the sum as
\begin{align*}
\sum_{\substack{|a|\leq B_1\\|b| \leq B_2 \\ \gcd(a, b) > Y}} 1
&= \sum_{Y < n \leq \min\set{B_1,B_2}} \sum_{\substack{|a|\leq B_1\\|b|\leq B_2\\\gcd(a, b)=n}} 1
= \sum_{Y < n \leq \min\set{B_1,B_2}} \sum_{\substack{|a|\leq \frac{B_1}{n}\\|b|\leq \frac{B_2}{n}\\\gcd(a, b)=1}} 1 \\
&\leq \sum_{Y < n \leq \min\set{B_1,B_2}} \prod_{i=1}^2\braces{\frac{2B_i}{n}+1}.
\end{align*}
Thanks to the bound $\frac{2B}{n} + 1 \leq \frac{3B}{n}$ for $n \leq B$, we then get that the above is at most $9B_1B_2 \sum_{Y < n \leq \min\set{B_1,B_2}} \frac{1}{n^2}$. Extending the sum to infinity and comparing with the integral $\int_{Y - 1}^\infty \frac{1}{x^2} dx$ gives the lemma.
\end{proof}

Given an integer $a \neq 0$, we define its squarefull part to be 
\[\text{sqfull}(a) = \prod_{\substack{p\mid a\\v_p(a)\geq 2}}p^{v_p(a)}.\]
For $a = 0$, we set $\text{sqfull}(a) := 0$.

\begin{lemma}
\label{lLargeSquare}
For $B,Y\geq 1$, we have
\[
\sum_{\substack{|a| \leq B \\ \textup{sqfull}(a) > Y}} 1 \ll \frac{B}{Y^{1/2}}.
\]
\end{lemma}

\begin{proof}
Note that we may once more assume that $Y \leq B$. We have
\[
\sum_{\substack{|a| \leq B \\ \textup{sqfull}(a) > Y}} 1 = 2\sum_{\substack{1 \leq a \leq B \\ \textup{sqfull}(a) > Y}} 1.
\]
We now rewrite the sum as
\[
\sum_{\substack{1 \leq a \leq B \\ \textup{sqfull}(a) > Y}} 1 = \sum_{\substack{Y < n \leq B \\ n \text{ squarefull}}} \sum_{\substack{1 \leq a \leq B \\ \textup{sqfull}(a) = n}} 1 \leq \sum_{\substack{Y < n \leq B \\ n \text{ squarefull}}} \sum_{\substack{1 \leq a \leq B \\ n \mid a}} 1 \leq \sum_{\substack{Y < n \leq B \\ n \text{ squarefull}}} \frac{B}{n}.
\]
Substituting $n=u^2v^3$ for integers $u,v$ such that $v$ is squarefree leads to the bound
\[\sum_{\substack{Y < n \leq B \\ n \text{ squarefull}}} \frac{1}{n}\leq \sum_{1\leq v\leq B^{1/3}}\frac{1}{v^3}\sum_{\sqrt{\frac{Y}{v^3}}< u\leq \sqrt{\frac{B}{v^3}}}\frac{1}{u^2}\ll Y^{-1/2}.\]
\end{proof}

\begin{lemma}
\label{lem:PowerOmegaBound}
    Let $a\in\ZZ_{\geq 1}$, then 
    \[\sum_{1 \leq d \leq X} \mu^2(d) a^{\omega(d)} \ll_a X\log(X)^{a-1}\]
    uniformly for $X > 2$.
\end{lemma}

\begin{proof}
    Note that, for squarefree $d$, $a^{\omega(d)}$ is equal to the $a$-th divisor function $\tau_{a}(d)$.  The bound on this quantity is classical (see e.g. \cite{IK}*{Ch. 1, Exercise 2}).
\end{proof}

\subsubsection{Statement of the bilinear sieve}
We will require a version of the large sieve which holds over number fields. To state it, we need the following definition \cite{San23}*{Def.~4.19}. Given a number field we denote its monoid of ideals by $I_K$.

\begin{definition}\label{def:osc bilin char}
Let $K, L$ be number fields, $A > 0$ and $q \in \N$, an $(A,q)$-\emph{oscillating bilinear character} is a bilinear morphism $\gamma: I_K \times I_L \to \C$ such that for all $I \in I_K, J \in I_{L}$ the character $\gamma(I, \cdot)$, respectively $\gamma(\cdot, J)$
\begin{enumerate}
\item is either $0$ or a Hecke character of finite modulus $\mathfrak{m}$ such that $\Norm_L(\mathfrak{m}) \leq q \Norm_K(I)^{A}$, respectively $\Norm_K(\mathfrak{m}) \leq q \Norm_L(J)^{A}$,
\item and is non-principal if there exists a $p \nmid q$ such that $v_p(\Norm_K(I)) = 1$, respectively $v_p(\Norm_L(J)) = 1$.
\end{enumerate}
\end{definition}

The following large sieve inequality holds for such bilinear characters \cite{San23}*{Prop.~4.21}, in the notation of loc.~cit. we choose $U = I_K \times I_L$ since this is the only case we will need. A similar sieve result has been independently proven in \cite{KR}*{Prop.~4.3}.

\begin{proposition}
\label{prop:large_sieve}
Let $K, L$ be number fields. Let $a_{I}, b_{J}$ be complex numbers of absolute value at most $1$, where $I,J$ range over $I_K, I_L$ respectively. Let $\gamma: I_K \times I_L \to \C$ be an $(A, q)$-oscillating bilinear character. There exists $\delta > 0$ depending only on $K, L, A$ such that for all $M, N \geq 2$
\begin{equation*}
\mathop{\sum\sum}_{\substack{\Norm(I) \leq M \\ \Norm(J) \leq N}} a_{I} b_{J} \gamma(I, J) \ll_{K, L, A} qMN(\log M N)^{\delta^{-1}}(M^{-\delta} + N^{-\delta}).
\end{equation*}
\end{proposition}

It remains to show that the products of power residue symbols we encounter are indeed oscillating bilinear characters.

\begin{lemma}
\label{lem:power_residue_symbols_oscillate}
Let $d, e \in \N$ such that $(d, e) \neq (1,1), (2,2)$ and $\varepsilon_d \in (\Z/d \Z)^{\times}$, $\varepsilon_e \in (\Z/e \Z)^{\times}$. Then there exists an $\alpha>0$ depending only on $d$ and $e$ such that
\[
\gamma: I_{\Q(\zeta_d)} \times I_{\Q(\zeta_e)} \to \C: (I, J) \mapsto \left( \frac{\Norm(J)}{I}\right)_{\Q(\zeta_d), d}^{\varepsilon_d} \left(\frac{\Norm(I)}{J}\right)_{\Q(\zeta_e), e}^{\varepsilon_e}
\]
is an $\braces{\alpha, (de)^{\alpha}}$-oscillating bilinear character. Moreover, if $p$ is a prime such that $p\mid \Norm(I)$ with multiplicity $1$ and $p\nmid de$, then $p$ divides the norm of the conductor of $\gamma(I,\cdot)$.
\end{lemma}

\begin{proof}
For $I \in I_{\Q(\zeta_d)}$ we have that $\gamma(I, \cdot)$ is the product of a Hecke character modulo $\Norm(I) \in \Z[\zeta_e]$ and a Hecke character on the Galois group of $\QQ(\zeta_e, \sqrt[e]{\Norm(I)})/\QQ(\zeta_e)$. The norm of its conductor thus divides $(e \Norm(I))^{O_{e}(1)}$. The corresponding property for $J \in I_{\Q(\zeta_e)}$ and $\gamma(\cdot, J)$ holds by symmetry. It remains to show non-principality.

By symmetry and the same argument as in \cite{San23}*{Lemma~4.20} it suffices to show the following: for all prime ideals $\mathfrak{p} \subset \Z[\zeta_d]$ coprime to $de$ and with prime norm $p := \Norm(\mathfrak{p})$, there exists an ideal $J \subset \Z[\zeta_e]$ such that $\gamma(\mathfrak{p}, J) \neq 1$. 

If $e=1$ and $d>1$, the existence is clear: choose $J\in\ZZ$ coprime to $p$ whose image mod $p$ is not in $\FF_p^{\times d}$ (since $p$ totally splits in $\ZZ[\zeta_d]$, such $J$ exists). If $e=2$ and $d\neq 2$, then by Dirichlet's theorem we choose $J$ to be a prime number $J\equiv 1 \bmod 4$ whose image mod $p$ generates $\FF_p^\times/\FF_p^{\times 2d}$. By quadratic reciprocity
\[\gamma(\fp,J)=\art{J}{\fp}_{\QQ(\zeta_d),d}^{\varepsilon_d}\art{J}{p}_{2}=-\zeta_d\neq 1\]
for some primitive $d$-th root of unity $\zeta_d$.

If $e> 2$, then choose a rational prime $q$ whose Frobenius in $\Gal(\QQ(\zeta_e,\sqrt[e]{p})/\QQ)$ is a generator of $\Gal(\QQ(\zeta_e,\sqrt[e]{p})/\QQ(\zeta_e))$, which exists by Chebotarev's density theorem. Then, fixing a choice of prime ideal $\mathfrak{q} \subset \ZZ[\zeta_e]$ such that $\mathfrak{q}\mid q$, for each $\sigma\in \Gal(\QQ(\zeta_e)/\QQ)$ we have (for some fixed $e$-th root of unity $\zeta_e$)
\[
\gamma(\mathfrak{p},\sigma(\mathfrak{q}))=\braces{\frac{q}{\mathfrak{p}}}_{\QQ(\zeta_d),d}^{\epsilon_d}\braces{\frac{p}{\sigma(\mathfrak{q})}}_{\QQ(\zeta_e),e}^{\epsilon_e}=\braces{\frac{q}{\mathfrak{p}}}_{\QQ(\zeta_d),d}^{\epsilon_d}\sigma(\zeta_e)^{\epsilon_e}.
\]
In particular, there exists a choice of $\sigma$ (equivalently, a choice of prime ideal above $q$) such that this expression is not $1$.
\end{proof}

\subsubsection{Statement of Siegel--Walfisz}
We shall also require a variant of the Siegel--Walfisz theorem.

\begin{lemma}\label{lem:SiegelWalfisz}
Let $\chi:I_{\Z[\zeta_d]}\rightarrow \CC$ be a primitive, non-trivial Hecke character of conductor $f_{\chi}$. Then uniformly for $N(f_{\chi}) \ll_C (\log B)^C$, we have 
\begin{equation}\label{sum over split primes}
\sum_{\substack{\mathfrak{p} \subset \Z[\zeta_d] \textnormal{ split prime} \\ \Norm(\mathfrak{p}) \leq B}} \chi(\mathfrak{p}) \ll_C B (\log B)^{- C}.
\end{equation}
\end{lemma}

\begin{proof}
This follows from \cite{Goldstein}.
\end{proof}

We then apply this with the LSD method to obtain the following bound.

\begin{lemma}
\label{lem:LSDGeneral}
Let $C>0$.  For every $B>1$, every $1 \leq \Delta \leq B^C$, and all non-principal Hecke characters $\chi: I_{\Z[\zeta_d]} \to \C$ with modulus $\fm$ and conductor $f_{\chi}\mid \fm$ such that $\Norm(f_{\chi}) \omega(\Norm(\mathfrak{m})) \leq (\log B)^C$ we have the bound
\begin{equation}
\label{eq:Siegel-Walfisz_older}
\sum_{\substack{I \in I_{\Z[\zeta_d]} \\ \Norm(I) \leq B}} \mu^2(\Norm(I) \cdot \Delta) \chi(I) \ll B (\log B)^{-C}
\end{equation}
where the implied constant depends only on $C$ and $d$.
\end{lemma}

\begin{proof}
    We may and will assume that $\Delta$ is squarefree, as the sum vanishes otherwise. Consider the function $h:\mathbb{N}\to\CC$ given by
    \[h_\Delta(m) = \mu^2(m \Delta) \sum_{\substack{I\in I_{\ZZ[\zeta_d]}\\\Norm(I)=m}}\chi(I).\]
    This is a multiplicative function by unique factorization of ideals in $\Z[\zeta_d]$ and the left-hand side of \eqref{eq:Siegel-Walfisz_older} is $\sum_{m \leq B} h_\Delta(m)$. We will apply a version of the LSD method \cite{Kou}*{Rmk.~13.3} with $\kappa = 0$, $k = d$, $\epsilon = 1$ and $Q = \exp((\log B)^{1/2})$. Then \cite{Kou}*{Rmk.~13.3} gives
    $$
    \sum_{m \leq B}h_\Delta(m) \ll_C B(\log B)^{-C}
    $$
    so it remains to check that the hypotheses \cite{Kou}*{(13.1), (13.2)} hold. For \cite{Kou}*{(13.2)}, we remark that $h_\Delta$ is multiplicative, supported on squarefree integers and satisfies $|h_\Delta(p)| \leq d$. Hence it follows that $|h_\Delta(m)| \leq \tau_d(m)$. Therefore it remains to check that
    \begin{equation}
        \sum_{p\leq x} h_\Delta(p)\log(p) \ll x \log(x)^{-C} \label{eq:Kou131}
    \end{equation}
    for $x \geq Q$, where the implied constant is uniform in $\Delta$ but depends on $C$ and $d$. Writing $h=h_1$, note that by our upper bound on $\Delta$ we have
    \[\abs{\sum_{p\leq x}h_\Delta(p)\log(p)-\sum_{p\leq x}h(p)\log(p)}\leq d\omega(\Delta)\log(x)\ll_{C,d} \log(x)^3.\]
    Hence it is sufficient to prove the inequality for $h$.  We begin by considering the sum of $h(p)$, without the log weighting. Let $\chi'$ be the primitive Hecke character determined by $\chi$.  Note that
    \[
        \left|\sum_{\substack{\mathfrak{p} \subset \Z[\zeta_d] \text{ split prime} \\ \Norm(\mathfrak{p}) \leq x}} \chi(\mathfrak{p}) - \sum_{\substack{\mathfrak{p} \subset \Z[\zeta_d] \text{ split prime} \\ \Norm(\mathfrak{p}) \leq x}} \chi'(\mathfrak{p})\right| \leq \sum_{\mathfrak{p} \mid  \mathfrak{m}} 1 = O(\omega(\Norm(\mathfrak{m}))).
    \]
    Using the assumption that $\Norm(f_\chi) \omega(N(\mf{m})) \ll_C (\log B)^C \ll_C (\log x)^{2C}$, we deduce from \Cref{lem:SiegelWalfisz} with $2C$ that
    \begin{equation}\label{eq:sum hp}
        \sum_{p \leq x} h(p) \ll_{C, d} \omega(N(\mf{m})) + x (\log x)^{-C} \ll_C x(\log x)^{-C}.
    \end{equation}
    Having established this bound for $\sum_{p \leq x} h(p)$, the lemma follows from partial summation.
\end{proof}

\noindent In our particular situation, we deduce the following.

\begin{lemma}
\label{lem:Siegel-Walfisz}
Let $C>0$.  For every $B>1$, every $1\leq\Delta\leq X^C$, and all Hecke characters $\chi: I_{\Z[\zeta_d]} \to \C$ with modulus $\fm$ and conductor $f_{\chi}\mid \fm$ such that $\Norm(f_{\chi}) \omega(\Norm(\mathfrak{m})) \leq (\log B)^C$ and there exists a prime $p \nmid d$ such that $p \mid \Norm(f_{\chi})$, we have the bound
\begin{equation}
\label{eq:Siegel-Walfisz}
\sum_{\substack{I \in I_{\Z[\zeta_d]} \\ \Norm(I) \leq B}} \mu^2(\Norm(I) \cdot \Delta) g(\Norm(I)) \chi(I) \ll_C B (\log B)^{-C}
\end{equation}
where $g$ is the multiplicative function from \Cref{eq:mult function g} and the implied constant depends only on $C$ and $d$.
\end{lemma}

\begin{proof}
The value $g(\Norm(\mathfrak{p}))$ only depends on the value of $\Norm(\mathfrak{p})$ modulo $d$ so can be written as a linear combination of Hecke characters $\psi$ of modulus $d$. The character $\psi \chi$ has modulus $d \mathfrak{m}$ and its conductor $f_{\psi \chi}$ divides $d f_\chi$. The condition that there exists $p \mid \Norm(f_\chi)$ such that $p \nmid d$ ensures that $p\mid f_{\psi\chi}$, so in particular $\psi \chi$ is non-principal. Thus each $\psi\chi$ satisfies the hypotheses of \Cref{lem:LSDGeneral}, and so the claim is immediate by summing over the (finitely many, depending only on $d$) characters $\psi$.
\end{proof}

\subsection{Application of oscillation results}
We split our sum up into subsums depending on which variables are medium-sized (meaning larger than a power of $\log$ as-yet to be chosen).

\begin{definition}
 Let $X:=\max\set{B_1,B_2,B_3}$, and write
 \[\mathcal{I} := \set{(\mathbf{v}, d) : \mathbf{v} \in \set{\be_1, \be_2, \be_3}, d \mid n}.\]
 Then for each $\cJ\subseteq \cI$, let $N_{\textup{f}}^{\cJ}(\bB;\bs,\bb,M)$ be the subsum of $N_{\textup{f}}(\bB;\bs,\bb,M)$ such that
 \begin{itemize}
 \item for indices $(\bv,d)\not\in\cI$, we have $\Norm(I_{\bv,d})\leq \log(X)^{2C}$;
 \item for indices $(\bv,d)\in\cI\backslash\cJ$, we have $\Norm(I_{\bv,d})\leq\log(X)^A$;
 \item for indices $(\bv,d)\in\cJ$, we have $\Norm(I_{\bv,d})>\log(X)^A$.
 \end{itemize}
 Here $A,C>1$ are large constants to be chosen later with $A$ much larger than $C$. 
\end{definition}

We begin by expressing our sum in terms of these subsums.

\begin{definition}
    We say that $(\mathbf{v}, d), (\mathbf{w}, e) \in \cB\times \set{f \mid n}$ are \emph{unlinked} if at least one of the following three conditions hold:
    \begin{enumerate}
        \item both $(\mathbf{v}, d), (\mathbf{w}, e) \notin \mc{I}$;
        \item $\bw \cdot f(\bv) \equiv 0 \bmod d$ and $\bv \cdot f(\bw) \equiv 0 \bmod e$;
        \item $d,e$ are both even and $\bw \cdot f(\bv) \equiv d/2 \bmod d$ and $\bv \cdot f(\bw) \equiv e/2 \bmod e$.
    \end{enumerate}
We say that $(\mathbf{v}, d)$ and $(\mathbf{w}, e)$ are \emph{linked} if they are not unlinked.
\end{definition}

\begin{remark}\label{rem:unlinked in I}
By definition, if $i\neq j$ then $f(\be_i)\cdot\be_j= \pm1$.  In particular, $(\be_i,d)$ and $(\be_j,e)$ are unlinked precisely if $d=e=1$ or $d=e=2$.
\end{remark}

\begin{lemma}
\label{lem:LinkedVariablesOscillate}
Assume that $(\mathbf{v}, d), (\mathbf{w}, e) \in \cB\times \set{f \mid n}$ are linked. Then there exists $\alpha>0$ depending only on $d$ and $e$ such that
\[
\gamma: I_{\Q(\zeta_d)} \times I_{\Q(\zeta_e)} \to \C: (I_{\mathbf{v}, d}, I_{\mathbf{w}, e}) \mapsto \left(\frac{\Norm(I_{\mathbf{w}, e})}{I_{\mathbf{v}, d}}\right)_{\Q(\zeta_d), d}^{\mathbf{w} \cdot f(\mathbf{v})} \left(\frac{\Norm(I_{\mathbf{v}, d})}{I_{\mathbf{w}, e}}\right)_{\Q(\zeta_e), e}^{\mathbf{v} \cdot f(\mathbf{w})}
\]
is an $\braces{\alpha, (de)^{\alpha}}$-oscillating bilinear character. Moreover, if $p$ is a prime such that $p \mid \Norm(I_{\mathbf{v}, d})$ with multiplicity $1$ and $p \nmid de$, then $p$ divides the norm of the conductor of $\gamma(I_{\mathbf{v}, d}, \cdot)$.
\end{lemma}

\begin{proof}
We wish to apply Lemma \ref{lem:power_residue_symbols_oscillate}. Define $d'$ to be the largest divisor of $d$ such that $\mathbf{w} \cdot f(\mathbf{v})$ becomes a unit modulo $d'$, and similarly define $e'$ to be the largest divisor of $e$ such that $\mathbf{v} \cdot f(\mathbf{w})$ is a unit modulo $e'$. Note that our definition of linked corresponds precisely to the condition that $(d', e') \not \in \{(1, 1), (2, 2)\}$. Therefore we may indeed apply Lemma \ref{lem:power_residue_symbols_oscillate} with $(d', e')$, and the current lemma follows.
\end{proof}

\begin{lemma}
\label{lem:splitting Nf as NfM}
We have
\[
N_{\textup{f}}(\bB;\bs,\bb,M) = \sum_{\cJ\subseteq \cI}N_{\textup{f}}^{\cJ}(\bB;\bs,\bb,M) + O\braces{\frac{B_1B_2B_3}{\log(X)^{C}}},
\]
where the implied constant is independent of $\bB$.
\end{lemma}

\begin{proof}
The only difference between the main term on the right and the sum of interest on the left is the subsum where for some $\bv\in\cB\backslash\set{e_1, e_2, e_3}$ and $d \mid n$ we have $\Norm(I_{\bv,d})> \log(X)^{2C}$. Translating back to the $(a_1,a_2,a_3)$ variables via \Cref{lem:initialCOV}, and writing $Z:=\log(X)^{2C}$, we can trivially bound this subsum by
\[\sum_{\substack{\ba\leq \bB\\\exists i\textnormal{ s.t. }\\\textnormal{sqfull}(a_i)>Z}}1
+
\sum_{\substack{\ba\leq \bB\\\exists i,j\textnormal{ s.t. }\\\textnormal{gcd}(a_i,a_j)>Z}}1
\ll
\frac{B_1B_2B_3}{\sqrt{Z}},
\]
where the bound is by \Cref{lLargeSquare,lLargeGCD}.
\end{proof}

\subsubsection{Linked indices in \texorpdfstring{$\cJ\subseteq \cI$}{J contained in I}}
In the situations where two linked indices are contained in $\cJ$, we obtain cancellation via the large sieve.

\begin{lemma}
\label{lem:two medium linked}
Let $\cJ\subseteq \cI$ be a subset which contains two linked indices. Then there exists $\delta > 0$, independent of $A,C$, such that 
\[
N^{\cJ}_{\textup{f}}(\mathbf{B}; \mathbf{s}, \mathbf{b}, M) \ll B_1B_2B_3\log(X)^{\delta^{-1}-\delta A+C(\#\cB-3)\tau(n)+3n\tau(n)}.
\]
\end{lemma}

\begin{proof}
We may assume by symmetry that our linked indices are $(\be_1,d),(\be_2,e)\in\cJ$. We encode the coprimality of $\Norm(I_{\be_1, d})$ and $\Norm(I_{\be_2, e})$ via the power residue symbols being non-zero. It remains to bound the following sum
\begin{multline*}
\sum_{(x_{i, \bullet})_{1 \leq i \leq 3}} 
\sum_{\substack{(I_{\mathbf{u}, c})_{\mathbf{u} \in \mathcal{B}, c \mid n, (\mathbf{u}, c) \not\in\set{ (\be_1, d), (\be_2, e)}} \\ I_{\mathbf{u}, c} \subseteq \Z[\zeta_c]}} \\
\Bigg|
\sum_{\substack{I_{\be_1, d} \subseteq \Z[\zeta_d] \\ \Norm(I_{\be_1, d}) > (\log X)^A \\ \Norm(I_{\be_1, d}) \leq T_1 \\ \Norm(I_{\be_1, d}) \equiv c_1 \bmod M}}
\sum_{\substack{I_{\be_2, e} \subseteq \Z[\zeta_e] \\ \Norm(I_{\be_2, e}) > (\log X)^A \\ \Norm(I_{\be_2, e}) \leq T_2 \\ \Norm(I_{\be_2, e}) \equiv c_2 \bmod M}}
\alpha(I_{\be_1, d}) \beta(I_{\be_2, e})
\left(\frac{\Norm(I_{\be_2, e})}{I_{\be_1, d}}\right)_{\Q(\zeta_d), d}^{\be_2 \cdot f(\be_1)} \left(\frac{\Norm(I_{\be_1, d})}{I_{\be_2, e}}\right)_{\Q(\zeta_e), e}^{\be_1 \cdot f(\be_2)}
\Bigg|,
\end{multline*}
where
\begin{align*}
&T_1 := \frac{B_1}{\left|x_{1, \bullet} \prod\limits_{\substack{\mathbf{u} \in \mathcal{B}, c \mid n \\ (\mathbf{u}, c) \neq (\be_1, d)}} \Norm(I_{\mathbf{u}, c})^{\pi_1(\mathbf{u})}\right|} & T_2 := \frac{B_2}{\left|x_{2, \bullet} \prod\limits_{\substack{\mathbf{u} \in \mathcal{B}, c \mid n \\ (\mathbf{u}, c) \neq (\be_2, e)}} \Norm(I_{\mathbf{u}, c})^{\pi_2(\mathbf{u})}\right|}
\end{align*}
and where $\alpha(I_{\be_1, d})$ and $\beta(I_{\be_2, e})$ are complex numbers depending only on respectively $I_{\be_1, d}$ and $I_{\be_2, e}$ (and the variables in the outer sum). Here $c_1, c_2$ range over all residue classes modulo $M$ satisfying the conditions
\begin{align*}
&c_1 x_{1, \bullet} \prod_{\substack{\mathbf{u} \in \mathcal{B}, c \mid n \\ (\mathbf{u}, c) \neq (\be_1, d)}} \Norm(I_{\mathbf{u}, c})^{\pi_1(\mathbf{u})} \equiv b_1 \bmod M, & c_2 x_{2, \bullet} \prod\limits_{\substack{\mathbf{u} \in \mathcal{B}, c \mid n \\ (\mathbf{u}, c) \neq (\be_2, e)}} \Norm(I_{\mathbf{u}, c})^{\pi_2(\mathbf{u})} \equiv b_2 \bmod M.
\end{align*}
Observe that $T_1$ and $T_2$ only depend on $B_1$ and $B_2$ and the variables in the outer sum. We may assume that $T_1, T_2 \geq (\log X)^A$ since the sum is otherwise equal to $0$. By Lemma \ref{lem:LinkedVariablesOscillate} we may apply Proposition \ref{prop:large_sieve} to the inner sum over $I_{\be_1, d}$ and $I_{\be_2, e}$. This shows that the inner sum is $\ll_n T_1 T_2 (\log X)^{\delta^{-1} - \delta A}$ for some $\delta > 0$. 

Writing $\cI' = \cI\backslash\set{(\be_1,d),(\be_2,e)}$, we sum over the remaining terms and obtain 
\begin{align*}
N^{\cJ}_{\textup{f}}(\mathbf{B}; \mathbf{s}, \mathbf{b}, M)
&\ll_n B_1B_2\log(X)^{\delta^{-1}-\delta A+C(\#\cB-3)\tau(n)}\sum_{\substack{(I_{\mathbf{u},c})_{(\mathbf{u},c)\in\cI'}\\\prod_{c\mid n}\Norm(I_{e_j,c})\leq B_j}}\braces{\prod\limits_{\substack{(\mathbf{u},c)\in\cI'\\\mathbf{u}\in\set{\be_1,\be_2}}}\Norm(I_{\mathbf{u},c})}^{-1},\\
&\leq B_1B_2\log(X)^{\delta^{-1}-\delta A + C(\#\cB-3)\tau(n)}\sum_{\substack{N_1\leq B_1\\N_2\leq B_2\\N_3\leq B_3}}\frac{\mu^2(N_1N_2N_3) (n\tau(n))^{\omega(N_1N_2)}}{N_1N_2},\\
&\ll B_1B_2B_3\log(X)^{\delta^{-1}-\delta A+C(\#\cB-3)\tau(n)+3n\tau(n)},
\end{align*}
where in the final inequality we have applied \Cref{lem:PowerOmegaBound} and partial summation on the sums over $N_1$ and $N_2$.  
\end{proof}

In particular, choosing $A$ substantially larger than $C$ (depending on $\delta,n$) will allow us to draw these terms into the error.  We will consider $A$ to be substantially larger than $C$ and the $B_i$ to be of similar sizes, in the following sense.

\begin{assumption}
\label{ass:A big for C and Bi similar size}
Assume that 
\begin{itemize}
\item $\log\braces{\min\set{B_1,B_2,B_3}}\geq \log(X)^\epsilon$ for some $\epsilon>0$; and
\item $-C\geq \delta^{-1}-\delta A+C(\#\cB-3)\tau(n)+3n\tau(n)$ where $\delta$ is as in \Cref{lem:two medium linked}.
\end{itemize}
\end{assumption}

We have reduced to considering $\cJ\subseteq \cI$ which do not contain any pairs of linked indices (i.e. $\cJ$ which are unlinked). In fact we shall only be interested in such sets which contain all three basic vectors at least once, since failing to include one will lead to a negligible contribution to the main term in \Cref{lem:splitting Nf as NfM}.  We summarise with the following lemma.

\begin{lemma}
\label{lem:reduce to NfM1 and 2 after large sieve}
Under \Cref{ass:A big for C and Bi similar size}, writing $\cJ_d:=\set{(\be_i,d):i=1,2,3}$, we have
\[
N_{\textup{f}}(\bB;\bs,\bb,M)=\sum_{\substack{d\mid n\\d\in\set{1,2}}}N_{\textup{f}}^{{\cJ_d}}(\bB;\bs,\bb,M)+O\braces{\frac{B_1B_2B_3}{\log(X)^C}}.
\]
\end{lemma}

\begin{proof}
By \Cref{lem:splitting Nf as NfM}, within the proposed error we can reduce the left hand side to $\sum_{\cJ\subseteq\cI}N_{\text{f}}^\cJ(\bB;\bs,\bb,M)$.  By \Cref{lem:two medium linked}, and our assumption on $A$ being substantially larger than $C$, we reduce our sum to being over unlinked subsets
\[\sum_{\cJ\subseteq\cI}N_{\text{f}}^\cJ(\bB;\bs,\bb,M)
=\sum_{\substack{\cJ\subseteq \cI\\\cJ\textnormal{ unlinked}}}N_{\textup{f}}^{\cJ}(\bB;\bs,\bb,M)+O\braces{\frac{B_1B_2B_3}{\log(X)^{C}}}.
\]
Now consider an unlinked set $\cJ\subseteq \cI$ such that there is some $i\in\set{1,2,3}$ satisfying $(\be_i,d)\not\in\cJ$ for all $d\mid n$. By symmetry assume that $i = 1$, which implies in our original variables (see \Cref{lem:Ca to Cb}) that $|a_1|$ is at most $C'(n) \log(X)^{C(\#\cB-3)\tau(n)+A\#(\cI\backslash\cJ)}$ with $C'(n)$ a function of $n$ only. Then the trivial bound, together with the first part of \Cref{ass:A big for C and Bi similar size}, gives
\begin{align*}
N_{\textup{f}}^\cJ(\bB;\bs,\bb,M) \ll_n B_2B_3\log(X)^{C(\#\cB-3)\tau(n)+A\#(\cI\backslash\cJ)} &\leq 
\frac{B_2B_3 \log(B_1)^{\frac{C(1+(\#\cB-3)\tau(n)) + A\#(\cI\backslash\cJ)}{\epsilon}}}{\log(X)^C}
\\&\ll_{A, C, \epsilon} \frac{B_1B_2B_3}{\log(X)^C}.
\end{align*}
Hence we can reduce the main term to summing over $\cJ$ which are unlinked and contain each basis vector $\be_i$ at least once.  Now, by definition of unlinked indices (see \Cref{rem:unlinked in I}), we must then have that $\cJ=\cJ_1$ or (if $n$ is even) $\cJ=\cJ_2$.  Thus the lemma holds.
\end{proof}

\subsection{Large index sets}
We have now reduced to the case that the only `large' variables are those indexed by elements in one of the $\cJ_d$.  The remaining variables are all still at least active, but are at most a power of $\log(X)$.  Our remaining goal is to reduce as many of these active variables to be $1$ as possible.  To do this, it will be convenient to understand what the maximal unlinked sets are which contain one of these $\cJ_d$.

\begin{definition}
 A \emph{large index set} is a maximal unlinked subset $\cM\subseteq \cB\times\set{d\mid n}$ which contains one of the sets $\cJ_d\subseteq \cI$ from \Cref{lem:reduce to NfM1 and 2 after large sieve}.  In other words, it is a subset which simultaneously satisfies: 
 \begin{itemize}
 \item every pair $(\bv,d),(\bw,e)\in\cM$ is unlinked, and $\cM$ is maximal with respect to inclusion;
 \item there exists $d\in\set{1,2}$ such that $\cJ_{d}\subseteq \cM$.
 \end{itemize}
\end{definition}
\noindent Such subsets are classified by the following lemma.  In particular, there is exactly one containing each $\cJ_d$.

\begin{lemma}
 \label{lem:unlinked sets}
 ${\cM}_1:=\set{(\bv,1)~:~\bv\in\cB}$ is a large index set. Moreover, if $n$ is odd then this is the only large index set, else if $n$ is even then there is one more large index set given by
 \[{\cM}_2:=\set{(\bv,d_\bv)~:~\bv\in\cB},\]
 where $d_\bv=\begin{cases}
 1&\textnormal{if }\bv\equiv \mathbf{0} \bmod 2;\\
 2&\textnormal{else.}
 \end{cases}$
\end{lemma}

\begin{proof}
    Let $\mc{M}$ be a large index set, and suppose that $n$ is odd. We recall that if at least one of $(\bv,d), (\bw,e)$ is in $\mc{I}$, then they are unlinked if     \begin{equation}\label{oddunlinked}
        \bw\cdot f(\bv) \equiv 0 \bmod d, \qquad \bv\cdot f(\bw) \equiv 0 \bmod e.
    \end{equation}
Suppose that $(\be_i,d_i), (\be_j,d_j) \in \mc{M}$ for $i,j$ distinct. Then
$\be_i\cdot f(\be_j) = \be_j\cdot f(\be_i) = \pm 1$, and so (\ref{oddunlinked}) implies that $d_i = d_j = 1$. Since a large index set must feature every basis vector $\be_1, \be_2, \be_3$, we have    \begin{equation}\label{mcapi}
        \mc{M} \cap \mc{I} = \{(\be_1,1),(\be_2,1),(\be_3,1)\}.
    \end{equation}
    Suppose that  $\mc{M}\backslash \mc{I}$ contains the element $(\bv,d)$. Then $(\bv,d)$ is unlinked with $(\be_i, 1)$ for all $i$, which means that $\be_i\cdot f(\bv) \equiv 0 \bmod{d}$. Since the image $f(\mc{B})$ lies in $\{0,\pm 1\}^3\backslash\{(0,0,0)\}$ it follows that $d=1$. Therefore, $\mc{M} \subseteq \mc{M}_1$. The property (\ref{oddunlinked}) is trivially satisfied for every pair of elements in $\mc{M}_1$, hence $\mc{M}_1$ is unlinked. We conclude by maximality that $\mc{M} = \mc{M}_1$ is the only large index set. 

    Now suppose that $n$ is even, and $(\be_i, d_i), (\be_j,d_j) \in \mc{M}$ with $i,j$ distinct. Then either (\ref{oddunlinked}) holds, which implies that $d_i=d_j=1$, or $d_i, d_j$ are both even and 
\begin{equation}\label{evenunlinked}
        \be_j \cdot f(\be_i) \equiv d_i/2 \bmod d_i, \qquad \be_i \cdot f(\be_j) \equiv d_j/2 \bmod d_j,
    \end{equation}
which implies that $d_i=d_j=2$. If (\ref{mcapi}) holds, we may argue as above to deduce that $\mc{M} = \mc{M}_1$. The remaining case to consider is 
\begin{equation}\label{mcapieven}
        \mc{M} \cap \mc{I} = \{(\be_1,2),(\be_2,2),(\be_3,2)\}.
    \end{equation}
    Since $f(\mc{B})\subseteq \{0,\pm 1\}^3\backslash\{(0,0,0)\}$, we see that if $(\bv,d) \in \mc{M}\backslash\mc{I}$ then $d=1$ or $d=2$. We have 
    \begin{align*}
        (\mc{M}\cap \mc{I}) \cup \{(\bv,1)\} \text{ is unlinked } &\iff f(\be_i) \cdot \bv \equiv 0 \bmod{2} \text{ for all } i \\
        &\iff v_j - v_k \equiv 0 \bmod{2} \text{ for all } j,k \\
        &\iff \bv \equiv \bm{0} \bmod{2},
    \end{align*}
    where the last equivalence uses $\bv \in \cB$. For the case $d=2$ we have 
     \begin{align*}
        (\mc{M}\cap \mc{I}) \cup \{(\bv,2)\} \text{ is unlinked } &\iff f(\bv) \cdot \be_i \equiv f(\be_i)\cdot \bv \bmod 2 \text{ for all } i \\
    &\iff f(\bv)\cdot \be_i \equiv v_j - v_k  \bmod 2  \text{ for all } i,j,k \text{ distinct}\\
        &\iff \bv \not\equiv \bm{0} \bmod 2,
    \end{align*}
    with the last equivalence following from inspection of Table \ref{table1}. We conclude that $\mc{M}_2$ is unlinked, and that $\mc{M} \subseteq \mc{M}_2$, and so by maximality we have $\mc{M}= \mc{M}_2$. 
\end{proof}

Having reduced to understanding $N_{\textup{f}}^{\cJ}$ for $\cJ=\cJ_d$ we now study those subsums.  In particular, we will show that the variables not indexed by the related large index set can be reduced to being trivial by applying the Siegel--Walfisz result. 

\begin{definition}
Let $e\mid n$ with $e\in\set{1,2}$. Then define $N_{\textup{f, main}}^{\cM_e}(\bB;\bs,\bb,M)$ to be the subsum of $N_{\textup{f}}^{\cJ_e}(\bB;\bs,\bb,M)$ where for every $(\bv,d)\not\in\cM_e$, $I_{\bv,d}$ is the trivial ideal. That is,
\begin{multline*}
N_{\textup{f, main}}^{\cM_e}(\bB;\bs,\bb,M) = \sum_{(x_{i, \bullet})_{1 \leq i \leq 3}} \sum_{\substack{(I_{\mathbf{v}, d})_{\mathbf{v} \in \mathcal{B}, d \mid n} \\ I_{\mathbf{v}, d} \subseteq \Z[\zeta_d]}}^{\sharp} g\left(\prod_{\mathbf{v} \in \mathcal{B}} \prod_{d \mid n} \Norm(I_{\mathbf{v}, d})\right) \times \\ 
\prod_{\mathbf{v} \in \mathcal{B}} \prod_{d \mid n} \left(\frac{-\prod_{m = 1}^3 x_{m, \bullet}^{\pi_m(f(\mathbf{v}))} \prod_{\mathbf{w} \in \mathcal{B}} \prod_{e \mid n} \Norm(I_{\mathbf{w}, e})^{\mathbf{w} \cdot f(\mathbf{v})}}{I_{\mathbf{v}, d}}\right)_{\Q(\zeta_d), d}.
\end{multline*}
where $\sharp$ imposes the summation conditions
\begin{gather*}
|x_{i, \bullet} \prod_{\mathbf{v} \in \mathcal{B}} \prod_{d \mid n} \Norm(I_{\mathbf{v}, d})^{\pi_i(\mathbf{v})}| \leq B_i, \quad \gcd(\{s_i x_{i, \bullet} \prod_{\mathbf{v} \in \mathcal{B}} \prod_{d \mid n} \Norm(I_{\mathbf{v}, d})^{\pi_i(\mathbf{v})} : i \in \{1, 2, 3\}\}) = 1 \\
x_{i, \bullet} \prod_{\mathbf{v} \in \mathcal{B}} \prod_{d \mid n} \Norm(I_{\mathbf{v}, d})^{\pi_i(\mathbf{v})} \ n \text{-powerfree}, \quad x_{i, \bullet} \prod_{\mathbf{v} \in \mathcal{B}} \prod_{d \mid n} \Norm(I_{\mathbf{v}, d})^{\pi_i(\mathbf{v})} \equiv b_i \bmod M \\
p \mid \Norm(I_{\mathbf{v}, d}) \Rightarrow p \not \in S(n), \quad p \mid x_{i, \bullet} \Rightarrow p \in S(n), \quad \mu^2\left(\prod_{\mathbf{v} \in \mathcal{B}} \prod_{d \mid n} \Norm(I_{\mathbf{v}, d})\right) = 1 \\
\bv \in \set{\be_1,\be_2,\be_3} \Rightarrow \Norm(I_{\bv,e}) > \log(X)^A,\quad \bv \not \in \set{\be_1,\be_2,\be_3} \Rightarrow \Norm(I_{\bv,e}) \leq \log(X)^{2C} \\
(x_{i, \bullet})_{1 \leq i \leq 3} \in \mathcal{D}(n), \quad (\bv, d) \not \in \mathcal{M}_e \Rightarrow I_{\bv, d} = (1).
\end{gather*}
\end{definition}

\begin{lemma}
\label{lem:reduction to Nfmain}
Let $d\mid n$ with $d\in\set{1,2}$, and $\cM_d$ be as in \Cref{lem:unlinked sets}. Assume \Cref{ass:A big for C and Bi similar size}. Then for all $\eta>0$ we have
\[
N_{\textup{f}}^{\cJ_d}(\bB;\bs,\bb,M)=N^{\cM_d}_{\textup{f, main}}(\mathbf{B}; \mathbf{s}, \mathbf{b}, M) + O\braces{\frac{B_1B_2B_3}{\log(X)^{\eta}}},
\]
where the implied constant is independent of $\bB$.
\end{lemma}  

\begin{proof}
Let $d$ be as in the lemma statement. Let $(\bw,e)\not\in\cM_d$ be an index outside of our large set, and consider the subsum where $I_{(\bw,e)}\neq (1)$. Let $\bv\in\set{\be_1,\be_2,\be_3}$ be such that $\bv\cdot f(\bw)\neq 0$ (and so $\bw\cdot f(\bv)\neq 0$), and by symmetry assume that $\bv=\be_1$. We will show that this subsum is negligible. Indeed, by the triangle inequality and our assumptions on $A$ and $C$ it is sufficient to bound the following sum
\begin{multline*}
\sum_{(x_{i, \bullet})_{1 \leq i \leq 3}} 
\sum_{\substack{(I_{\mathbf{u}, c})_{\mathbf{u} \in \mathcal{B}, c \mid n, (\mathbf{u}, c) \neq (\mathbf{v}, d), (\mathbf{u}, c) \neq (\mathbf{w}, e)} \\ I_{\mathbf{u}, c} \subseteq \Z[\zeta_c]}} \\
\sum_{\substack{I_{\mathbf{w}, e} \subseteq \Z[\zeta_e] \\ 1<\Norm(I_{\mathbf{w}, e}) \leq (\log X)^A \\ \Norm(I_{\mathbf{w}, e}) \equiv c_2 \bmod M}}
\Bigg|
\sum_{\substack{I_{\mathbf{v}, d} \subseteq \ZZ \\ (\log X)^A < \Norm(I_{\mathbf{v}, d}) \leq T_1 \\ \Norm(I_{\mathbf{v}, d}) \equiv c_1 \bmod M\\\gcd(\Norm(I_{\bv,d}), \prod \Norm(I_{\bu,c}))=1}}
g(I_{\mathbf{v}, d})
\left(\frac{\Norm(I_{\mathbf{w}, e})^{\mathbf{w} \cdot f(\mathbf{v})}}{I_{\mathbf{v}, d}}\right)
\left(\frac{\Norm(I_{\mathbf{v}, d})^{\mathbf{v} \cdot f(\mathbf{w})}}{I_{\mathbf{w}, e}}\right)_{\Q(\zeta_e), e}
\Bigg|,
\end{multline*}
where we used that $d \in \{1, 2\}$ (so $\Q(\zeta_d) = \Q$) and where
\begin{align*}
&T_1 := \frac{B_1}{\left|x_{1, \bullet} \prod\limits_{\substack{\mathbf{u} \in \mathcal{B}, c \mid n \\ (\mathbf{u}, c) \neq (\mathbf{v}, d)}} \Norm(I_{\mathbf{u}, c})^{\pi_1(\mathbf{u})}\right|}.
\end{align*}
Here $c_1, c_2$ range over all residue classes modulo $M$ satisfying the conditions
\begin{align*}
&c_1 x_{1, \bullet} \prod_{\substack{\mathbf{u} \in \mathcal{B}, c \mid n \\ (\mathbf{u}, c) \neq (\mathbf{v}, d)}} \Norm(I_{\mathbf{u}, c})^{\pi_1(\mathbf{u})} \equiv b_1 \bmod M, & c_2 x_{2, \bullet} \prod\limits_{\substack{\mathbf{u} \in \mathcal{B}, c \mid n \\ (\mathbf{u}, c) \neq (\mathbf{w}, e)}} \Norm(I_{\mathbf{u}, c})^{\pi_2(\mathbf{u})} \equiv b_2 \bmod M.
\end{align*}
Note that by our summation conditions and assumptions on $A,C$, we have
\begin{equation}
\label{eq:T1BdForSWArg}
\frac{B_1}{\log(X)^{\beta A}}\ll T_1\ll B_1
\end{equation}
for some $\beta>0$ depending only on $n$. 

We will now bound the inner sum. By \Cref{lem:power_residue_symbols_oscillate}, since $\bv\cdot f(\bw)\neq 0$ and $(d,e)\neq (1,1),(2,2)$ the map $\chi_{I_{\bw,e}}:I_\QQ\to \CC$ given by 
\[
I_{\bv,d} \mapsto
\left(\frac{\Norm(I_{\mathbf{w}, e})^{\mathbf{w} \cdot f(\mathbf{v})}}{I_{\mathbf{v}, d}}\right)
\left(\frac{\Norm(I_{\mathbf{v}, d})^{\mathbf{v} \cdot f(\mathbf{w})}}{I_{\mathbf{w}, e}}\right)_{\Q(\zeta_e), e},
\]
is a Hecke character whose conductor $\mathfrak{q}$ has norm at most $\braces{e\Norm(I_{\bw,e})}^{O_{e}(1)}$ and whose modulus $\mathfrak{m}$ is supported on the primes dividing $de \Norm(I_{\bw,e})$. In particular, for suitably large $\eta>0$ (in terms of $n$)
\[
\Norm(\mathfrak{q})\omega(\Norm(\mathfrak{m}))\leq \log(X)^{\eta A}.
\]
Moreover, by loc. cit., since $I_{\bw,e}\neq (1)$ and $\Norm(I_{\bw,e})$ is not divisible by $p\mid n$, the conductor of $\chi_{I_{\bw,e}}$ is divisible by at least one prime coprime to $de$. Hence by \Cref{lem:Siegel-Walfisz}, uniformly for $I_{\bw,e}\neq(1)$ with $\Norm(I_{\bw,e})\leq \log(X)^A$, the inner sum is $\ll T_1\log(T_1)^{-\eta A}$. Note now that by \Cref{eq:T1BdForSWArg} and \Cref{ass:A big for C and Bi similar size}, for all $\gamma > 0$
\[
T_1 \gg \frac{B_1}{\log(X)^{A\beta}} \geq \frac{B_1}{\log(B_1)^{A\beta/\epsilon}} \gg_{\epsilon,n,A} B_1^{1/2}.
\]
In particular, $\log(T_1)^{-\eta A}\ll \log(B_1)^{-\eta A}\ll \log(X)^{-\eta A\epsilon}$, and so the inner sum is bounded by $\ll T_1\log(T_1)^{-\eta A}\ll B_1\log(X)^{-\eta}$. Repeating this process for all $(\bw,e)$ finishes the proof.
\end{proof}

\subsection{The main term}
To conclude, we have isolated the main term of $N_{\textup{f}}(\bB;\bs,\bb,M)$ in terms of the newly introduced sums $N_{\textup{f, main}}^{\cM_d}(\bB;\bs,\bb,M)$. In fact, since $d\in\set{1,2}$ and the ideals in $\ZZ[\zeta_d]=\ZZ$ correspond to positive integers, we can simplify these to 
\begin{multline*}
N_{\textup{f, main}}^{\cM_d}(\bB;\bs,\bb,M) 
= \sum_{(x_{i, \bullet})_{1 \leq i \leq 3}} 
\sum_{\substack{(z_\bv)_{\bv \in \cB} \\ z_\bv\in\ZZ_{>0}}}^{\sharp\sharp}
g\left(\prod_{\mathbf{v} \in \mathcal{B}} z_{\bv}\right) \times \\
\braces{\prod_{\bv\in\cB}\braces{\frac{-\prod_{m=1}^3x_{m,\bullet}^{\pi_m(f(\bv))}}{z_{\bv}}}_{e_{\bv}}}\braces{\prod_{\bv,\bw\in\cB}\braces{\frac{z_{\bw}}{z_{\bv}}}^{\bw\cdot f(\bv)}_{e_{\bv}}},
\end{multline*}
where the summation conditions in $\sharp\sharp$ are:
\begin{gather*}
\abs{x_{i, \bullet} \prod_{\bv\in\cB} z_{\bv}^{\pi_i(\mathbf{v})}} \leq B_i, 
\quad \gcd(\{s_i x_{i, \bullet} \prod_{\mathbf{v} \in \mathcal{B}} z_{\bv}^{\pi_i(\bv)} : i \in \{1, 2, 3\}\}) = 1
\\ x_{i, \bullet} \prod_{\bv\in\cB} z_\bv^{\pi_i(\mathbf{v})} \ n \text{-powerfree}, 
\quad x_{i, \bullet} \prod_{\mathbf{v} \in \mathcal{B}} z_{\bv}^{\pi_i(\mathbf{v})} \equiv b_i \bmod M
\\ p \mid z_{\bv} \Rightarrow p \not \in S(n), 
\quad p \mid x_{i, \bullet} \Rightarrow p \in S(n), 
\quad \mu^2\left(\prod_{\bv\in\cB} z_{\bv}\right) = 1,
\quad (x_{i, \bullet})_{1 \leq i \leq 3} \in \mathcal{D}(n)
\\ \bv\in\set{\be_1,\be_2,\be_3}\Rightarrow z_{\bv}>\log(X)^A,
\quad\bv\not\in\set{\be_1,\be_2,\be_3}\Rightarrow z_{\bv}\leq\log(X)^{2C}
\end{gather*}
and where 
$$
e_{\bv} = \begin{cases} 1 &\text{if } d = 1 \\ d_{\bv} &\text{if } d = 2. \end{cases}
$$
We record the conclusion of our work with character sums below.

\begin{proposition}
\label{prop:main term isolated}
Under \Cref{ass:A big for C and Bi similar size}, we have for every $C > 0$
\[
N_{\textup{f}}(\bB;\bs,\bb,M) = \sum_{\substack{d\mid n\\d\in\set{1,2}}} N_{\textup{f, main}}^{\cM_d}(\bB;\bs,\bb,M) + O_C\braces{\frac{B_1B_2B_3}{\log(X)^{C}}},
\]
where $N_{\textup{f, main}}^{\cM_d}$ is as above.
\end{proposition}

\begin{proof}
Immediate consequence of \Cref{lem:reduce to NfM1 and 2 after large sieve} and \Cref{lem:reduction to Nfmain}.
\end{proof}

The section now concludes be reverting the change of variables of \Cref{lem:initialCOV} to present our work so far in terms of the simpler variables $\ba$, showing that we have now accounted for the everywhere-locally soluble constraint.

\begin{theorem}
\label{thm:end of character sums}
Under \Cref{ass:A big for C and Bi similar size}, for every $C>0$,
\[
N_{\textup{f}}(\mathbf{B}; \mathbf{s}, \mathbf{b}, M) = (1+\mathbf{1}_{\textup{even}}(n))\sum_{\ss{\ba \in \ZZ^3\\ |a_i|\leq B_i \forall i}}^{\dag} g\l(\prod_{i=1}^3 a_i\r)+
O_C\braces{\frac{B_1B_2B_3}{\log(X)^{C}}},
\]
where $\dag$ imposes the summation conditions
\begin{align}
    &\gcd(\{s_ia_i: i \in \{1,2,3\}\}) = 1,\label{gcd cond}\\
    &\ba \equiv \bb \bmod{M},\label{cong cond on a}\\
    &\ba\ n\textnormal{-powerfree and }(v_p(a_i))_{i=1}^3\in  \cB\cup\set{(0,0,0)},\label{factorisation shape cond}\\
    &\ba \in \mc{D}(n) \subseteq \RR^3,\label{Dn cond}
\end{align}
and $g$ is the multiplicative function defined in \eqref{eq:mult function g}, which is given on prime powers $p^r$ by 
    \begin{equation*}
g(p^r) = 
\begin{cases}
1 & \textup{if } p \in S(n) \\
\max(\{k \mid n : p \equiv 1 \bmod k\})^{-1} & \textup{if } p \not \in S(n),
\end{cases}
\end{equation*}
and satisfies $g(-1)=1$.
\end{theorem}

\begin{proof}
    We make choices of $A,C$ according to \Cref{ass:A big for C and Bi similar size}, and then apply \Cref{prop:main term isolated}.  This reduces us to studying the sums $N_{\text{f, main}}^{\cM_d}$.  We begin by showing that $N_{\text{f, main}}^{\cM_1}(\bB;\bs,\bb,M)=N_{\text{f, main}}^{\cM_2}(\bB;\bs,\bb,M)$, which will explain the factor of $1+\mathbf{1}_{\text{even}}(n)$.  To see this, recall that
\begin{multline*}
N_{\textup{f, main}}^{\cM_2}(\bB;\bs,\bb,M) 
= \sum_{(x_{i, \bullet})_{1 \leq i \leq 3}} 
\sum_{\substack{(z_\bv)_{\bv \in \cB} \\ z_\bv\in\ZZ_{>0}}}^{\sharp\sharp}
g\left(\prod_{\mathbf{v} \in \mathcal{B}} z_{\bv}\right) \times \\
\braces{\prod_{\bv\in\cB}\braces{\frac{-\prod_{m=1}^3x_{m,\bullet}^{\pi_m(f(\bv))}}{z_{\bv}}}_{d_{\bv}}}\braces{\prod_{\bv,\bw\in\cB}\braces{\frac{z_{\bw}}{z_{\bv}}}^{\bw\cdot f(\bv)}_{d_{\bv}}}.
\end{multline*}
We shall show that the product of the Legendre symbols is constantly equal to $1$. To this end, observe that
\begin{align*}
\braces{\prod_{\bv\in\cB}\braces{\frac{-\prod_{m=1}^3x_{m,\bullet}^{\pi_m(f(\bv))}}{z_{\bv}}}_{d_{\bv}}}\braces{\prod_{\bv,\bw\in\cB}\braces{\frac{z_{\bw}}{z_{\bv}}}^{\bw\cdot f(\bv)}_{d_{\bv}}} &= \prod_{\bv \in \cB} \left(\frac{-\prod_{i = 1}^3 a_i^{\pi_i(f(\bv))}}{z_{\bv}}\right)_{d_{\bv}} \\
&= \prod_{p \not \in S(n)} \left(-\frac{a_2}{a_1}, -\frac{a_3}{a_1}\right)_p,
\end{align*}
where we have set $a_i = x_{i, \bullet} \prod_{\bw\in\cB} z_{\bw}^{\pi_i(\bw)}$ and where we used Lemma \ref{lem:indicator function} (with the $n$ of that lemma equal to $2$) in the last equality. Note that admissibility of $\bb$ ensures that $\left(-\frac{a_2}{a_1}, -\frac{a_3}{a_1}\right)_p=1$ for all $p\in S(n)$, and since $(a_i)_{1 \leq i \leq 3} \in \mathcal{D}(n)$ we obtain $\left(-\frac{a_2}{a_1}, -\frac{a_3}{a_1}\right)_\infty=1$.  Hence the claim follows from Hilbert reciprocity.

Now we simply revert the change of variables of \Cref{lem:initialCOV} to show
\[
N_{\textup{f, main}}^{\cM_1}(\bB;\bs,\bb,M) 
= \sum_{(x_{i, \bullet})_{1 \leq i \leq 3}} 
\sum_{\substack{(z_\bv)_{\bv \in \cB} \\ z_\bv\in\ZZ_{>0}}}^{\sharp\sharp}
g\left(\prod_{\mathbf{v} \in \mathcal{B}} z_{\bv}\right).
\] 
Note that $g(m)=g(\prod_{p\mid m}p)$, and the conditions of $\sharp\sharp$ convert directly to those in $\dag$ to almost obtain the claim, but with additional bounds on squarefull parts of $a_i$ and pairwise $\gcd$s.  As in \Cref{lem:splitting Nf as NfM}, we apply the elementary estimates from the beginning of the section to discard these constraints and obtain the claim.
\end{proof}

%% file: Section5.tex
\section{Leading term analysis}
\label{sLeading}
Via \Cref{thm:end of character sums}, we have reduced the problem of studying $N_{\textup{f}}(\mathbf{B}; \mathbf{s}, \mathbf{b}, M)$ to understanding the summation
\[
\sum_{\ss{\ba \in \ZZ^3\\ |a_i|\leq B_i \forall i}}^{\dag} g\l(\prod_{i=1}^3 a_i\r),
\]
where the conditions of $\dag$ and the multiplicative function $g$ are presented in \Cref{thm:end of character sums}. In this section we will obtain the asymptotic leading term for this sum, and so obtain Theorems \ref{tMain} and \ref{thm:leadingconstant}.

\subsection{Evaluating the first main sum}
Let $M_i = M/\gcd(b_i, M), b_i' = b_i/\gcd(b_i, M)$ and $a_i' = a_i/\gcd(b_i,M)$. The condition $a_i\equiv b_i \bmod{M}$ is equivalent to $a_i'\in \ZZ$ and $a_i' \equiv b_i' \bmod{M_i}$. The change of variables from $a_i$ to $a_i'$ alters the height condition to $|a_i'| \leq B_i/\gcd(b_i,M)$ for all $i$ and (\ref{gcd cond}) to 
\begin{equation}
\gcd(\{s_i\gcd(b_i,M)a_i': i \in \{1,2,3\}\}) = 1.\label{new gcd cond}
\end{equation}
In view of (\ref{new gcd cond}), we define the pair $(\mathbf{s}, \mathbf{b})$ to be admissible if $\mathbf{b}$ is admissible and 
$$
\gcd(\{s_i\gcd(b_i,M): i \in \{1,2,3\}\}) = 1. 
$$
Henceforth we assume that $(\mathbf{s}, \mathbf{b})$ is admissible, since otherwise $N_{\textup{f}, \text{main}}^{\cM_1}(\mathbf{B}; \mathbf{s}, \mathbf{b}, M) =0$.  

We detect the condition $a_i' \equiv b_i' \bmod{M_i}$ with Dirichlet characters $\chi_i$ modulo $M_i$. We define
\begin{equation*}
\bm{\chi}(\ba) := \prod_{i=1}^3 \chi_i(a_i). 
\end{equation*}
Define a function $h = h_{\bm{\chi}}:\N^3 \rightarrow \Z$ by
\begin{equation}
    h(\ba) =\bm{\chi}(\ba)g\left(\prod_{i=1}^3 a_i\right)\mathbf{1}_{\{(\ref{new gcd cond}), (\ref{factorisation shape cond}) \textrm{ hold}\}}(\ba), 
\end{equation}
where $\mathbf{1}$ denotes the usual indicator function. Note that $h$ is multiplicative in the sense that if two triples $(a_i)_{1 \leq i \leq 3}, (\widetilde{a_i})_{1\leq i \leq 3}$ are pairwise coprime, then 
$$
h((a_i\widetilde{a_i})_{1\leq i \leq 3}) = h((a_i)_{1 \leq i \leq 3})h((\widetilde{a_i})_{1 \leq i \leq 3}). 
$$
Let $\bm{\eps} \in \{\pm 1\}^3$ be such that $\epsilon_i = \op{sgn}(a_i)$.  Replacing $a_i'$ with $\epsilon_ia_i'$, we obtain 
\begin{equation}
\label{nicer form for NfMain}    
N_{\textup{f}, \text{main}}^{\cM_1}(\mathbf{B}; \mathbf{s}, \mathbf{b}, M) =  \sum_{\ss{\chi_i \bmod{M_i} \\ 1 \leq i \leq 3}} \prod_{i = 1}^3 \frac{\overline{\chi_i(b_i')}}{\phi(M_i)} \sum_{\bm{\eps} \in \{\pm 1\}^3 \cap \mc{D}(n)} \bm{\chi}(\bm{\eps}) \sum_{\ss{\ba' \in \N^3\\ a_i'\leq B_i/\gcd(b_i,M) \forall i}}h(\ba').\\ 
\end{equation}
The outer sums in (\ref{nicer form for NfMain}) are finite, so we focus on finding an asymptotic formula for the sum 
\[
\sum_{\ss{\ba\in \N^3\\ a_i\leq B_i/\gcd(b_i,M) \forall i}}h(\ba').
\]
This is achieved by first applying a generalisation of M\"{o}bius inversion to reduce to studying partial sums of a simpler multiplicative function $f$, which can then be understood via the LSD method. Specifically, we define
\begin{equation}\label{eq:completelymult}
f(\ba) = \bm{\chi}(\ba) \widetilde{g}\l(\prod_{i=1}^3 a_i\r),
\end{equation}
where $\widetilde{g}$ is the completely multiplicative function such that $\widetilde{g}(p) = g(p)$ for all primes $p$.

\subsubsection{Application of M\"{o}bius inversion}
In what follows, we write $\ba$ in place of $\ba'$ to ease notation. Note that $|h(\ba)| \leq 1$ for all $\ba \in \N^{3}$. Therefore, the sums
\begin{equation}
\label{def:H(s)}
H(\bm{s}) := \sum_{\ba \in \N^3} \f{h(\ba)}{\ba^{\bm{s}}}, \qquad F(\bm{s}) := \sum_{\ba \in \N^3} \f{f(\ba)}{\ba^{\bm{s}}}
\end{equation}
are absolutely convergent in the region 
$$ 
\mc{H}_1 := \{\bm{s} \in \C^{3}: \Re(s_i) > 1 \textrm{ for all } i\},
$$
where we use multi-index notation $\ba^{\bm{s}} = \prod_{i=1}^3 a_i^{s_i}$. Define $G(\bm{s}) = H(\bm{s})F(\bm{s})^{-1}$.

\begin{lemma}
\label{lem: analytic function G}
The function $G(\bm{s})$ extends to an analytic function in the region 
$$ 
\mc{H}_{1/2} := \{\bm{s} \in \C^{3}: \Re(s_i) > 1/2 \textup{ for all } i\}.$$
\end{lemma}

\begin{proof}
    Define $h_1(p) = h((p,1,1)), h_2(p) = h(1,p,1)$, and $h_3(p) = h(1,1,p)$, and similarly for $f_i$. We also define $\bm{p} = (p, p, p) \in \N^3$. Using multiplicativity of $f$ and $h$, we have 
    \begin{align*}
        G(\bm{s}) &= \prod_p\l(\sum_{\br \in (\Z_{\geq 0})^{3}}\f{h(\bm{p}^{\br})}{\bm{p}^{\br \cdot \bm{s}}}\r)\l(\sum_{\br \in (\Z_{\geq 0})^{3}}\f{f(\bm{p}^{\br})}{\bm{p}^{\br \cdot \bm{s}}}\r)^{-1}\\
        &= \prod_p\l(1 + \sum_{i=1}^3\f{h_i(p)}{p^{s_i}}+ O(p^{-2\min(\Re(s_i))})\r)\l(1+ \sum_{i=1}^3\f{f_i(p)}{p^{s_i}} + O(p^{-2\min(\Re(s_i))})\r)^{-1}.
    \end{align*}
For $\bm{s} \in \mc{H}_{1/2}$, we have $O(p^{-2\min(\Re(s_i))}) = O(p^{-1-\delta})$ for some $\delta > 0$, and so this error term only alters the Euler product by a non-zero constant factor. Moreover, away from finitely many primes, we have $h_i(p) = f_i(p) = g(p)$. We conclude that the Euler product defining $G(\bm{s})$ is absolutely convergent in the region $\mc{H}_{1/2}$. 
\end{proof} 

Define a function $K:\N^3\rightarrow \C$ such that the formula
\begin{equation}
\label{def:K(s)}
G(\bm{s}) = \sum_{\ba \in \N^3}\f{K(\ba)f(\ba)}{\ba^{\bm{s}}}
\end{equation}
holds for all $\bm{s} \in \C^{3}$; this condition uniquely defines $K$ by the existence and uniqueness of multiple Dirichlet series, together with the fact that $f(\ba) = 0$ only if $h(\ba) = 0$.

We have the following generalisation of M\"{o}bius inversion:

\begin{lemma}
\label{generalised Moebius}
Let $h, K$ be functions as above. Then for all $\ba \in \N^3$, we have
\begin{equation}
h(\ba) = f(\ba) \sum_{\bw \mid \ba} K(\bw).
\end{equation}
\end{lemma}

\begin{proof}
We have 
\begin{align*}
    \sum_{\ba \in \N^3}\f{f(\ba)}{\ba^{\bm{s}}}\l(\sum_{\bw \mid \ba}K(\bw)\r) &= \prod_{p} \l(\sum_{\br \in (\Z_{\geq 0})^{3}} \f{f(\bm{p}^{\br})}{p^{\br \cdot \bm{s}}}\sum_{\ss{\bt \in (\Z_{\geq 0})^{3}\\ t_i \leq r_i \forall i}} K(\bm{p}^{\bt})\r) \\
    &=\prod_{p}\l(\sum_{\bt \in (\Z_{\geq 0})^{3}}\f{K(\bm{p}^{\bm{t}})f(\bm{p}^{\bm{t}})}{p^{\bt\cdot \bm{s}}}\r)\prod_p \l(\sum_{\bu \in (\Z_{\geq 0})^{3}} \f{f(\bm{p}^{\bm{u}})}{p^{\bu\cdot \bm{s}}}\r)\\
    &=\l(\sum_{\ba \in (\Z_{\geq 0})^{3}}\f{K(\ba)f(\ba)}{\ba^{\bm{s}}}\r)\l(\sum_{\ba \in (\Z_{\geq 0})^{3}}\f{f(\ba)}{\ba^{\bm{s}}}\r)\\
    &= G(\bm{s}) F(\bm{s})\\
    &= H(\bm{s}).
\end{align*}
Recalling the definition for $H(\bm{s})$ from (\ref{def:H(s)}), we deduce that
$$
\sum_{\ba \in \N^3}\f{f(\ba)\sum_{\bw \mid \ba}K(\bw)}{\ba^{\bm{s}}}= H(\bm{s}) = \sum_{\ba \in \N^3}\f{h(\ba)}{\ba^{\bm{s}}}.
$$
By the uniqueness of multiple Dirichlet series, equality of the above sums in the region $\mc{H}_{1}$ implies equality of the individual coefficients for any fixed $\ba \in \N^3$. 
\end{proof}

Let $\bw=(w_1,w_2,w_3)\in\mathbb{N}^3$ and let
$$
\mc{R}_{\bm{B}, \bw} = \{\ba \in \N^3: |a_i|\leq B_i/w_i\gcd(b_i,M) \text{ for all } i\}
$$
and $\mc{R}_{\bm{B}} = \mc{R}_{\bm{B}, (1,1,1)}$. Using Lemma \ref{generalised Moebius} and the multiplicativity of $f$, we have 
\begin{align}
\label{moebius on h}
\sum_{\ba \in \mc{R}_{\bm{B}}} h(\ba) &= \sum_{\ba \in \mc{R}_{\bm{B}}}f(\ba)\sum_{\bw\mid\ba} K(\bw) = \sum_{\bw \in \mc{R}_{\bm{B}}} K(\bw)\sum_{\ba \in \mc{R}_{\bm{B},\bw}} f(\bw*\ba) \nonumber \\
&= \sum_{\bw \in \mc{R}_{\bm{B}}} K(\bw)f(\bw)\sum_{\ba \in \mc{R}_{\bm{B},\bw}} f(\ba),
\end{align}
where the multiplication $\bw*\ba$ is defined coordinatewise.

\subsubsection{Partial sums of \texorpdfstring{$f(\bz)$}{f(z)}}
We first consider a one-dimensional analogue $f(a):=\chi(a)\widetilde{g}(a)$ for a Dirichlet character $\chi$ modulo $M$.

\begin{lemma}
\label{sum of f(p)}
For all $A, Y\geq 1$, we have 
$$
\sum_{p\leq Y}f(p) = \f{\alpha_{\chi} Y}{\log Y} + O_A\l(\f{Y}{(\log Y)^A}\r)
$$
where 
$$
\alpha_{\chi}= \f{1}{\phi(M)}\sum_{r \in (\Z/M\Z)^{\times}} \f{\chi(r)}{\gcd(n, r-1)}.
$$
\end{lemma}

\begin{proof}
    Recall that $n\mid M$. 
    For $r \in (\Z/M\Z)^{\times}$, we have 
    $$ \max\{k\mid n: r \equiv 1 \bmod k\} = \gcd(n,r-1).$$
    Dividing up into residue classes modulo $M$ and applying the prime number theorem in arithmetic progressions, we obtain
    \begin{align*}
        \sum_{p\leq Y}f(p) &= \sum_{l\mid n}l^{-1}\sum_{\ss{p\leq Y, p\notin S(n), p \nmid w\\ l = \max\{k\mid n: p \equiv 1 \bmod{k}\}}}\chi(p)\\
        &=\sum_{l\mid n}l^{-1}\sum_{\ss{r \in (\Z/M\Z)^{\times}\\l= \max\{k\mid n: r \equiv 1 \bmod k\}}}\chi(r)\#\{p \leq Y: p \notin S(n), p \nmid w, p \equiv r \bmod{n}\}\\
        &= \l(\f{Y}{\log Y} + O_A\l(\f{Y}{(\log Y)^A}\r)\r)\f{1}{\phi(M)}\sum_{l\mid n}l^{-1}\sum_{\ss{r \in (\Z/M\Z)^{\times}\\ \gcd(n,r-1)=l}}\chi(r)\\
        &=\l(\f{Y}{\log Y} + O_A\l(\f{Y}{(\log Y)^A}\r)\r)\alpha_{\chi}.
    \end{align*}
\end{proof}

\begin{corollary}
\label{sum of f(p)log(p)}
    For all $Y \geq 1$, we have 
    \[
    \sum_{p \leq Y} f(p)\log(p) = \alpha_{\chi}Y + O\l(\f{Y}{\log Y}\r).
    \]
\end{corollary}

\begin{proof}
    Applying Lemma \ref{sum of f(p)} and partial summation, we obtain
    \begin{align*}
        \sum_{p \leq Y}f(p)\log(p) &= \sum_{p \leq Y}f(p)\log Y - \int_{2}^Y \sum_{p\leq t}f(p)t^{-1}\op{d}t\\
        &=\alpha_{\chi}Y - \int_{2}^Y \l(\f{\alpha_{\chi}}{\log t} + O_A\l(\f{1}{(\log t)^{A}}\r)\r)\op{d}t + O_A\l(\f{Y}{(\log Y)^A}\r)\\
        &=\alpha_{\chi}Y + O\l(\f{Y}{\log Y}\r),
    \end{align*}
    the last line following from choosing $A=1$ and applying the bound $\mathrm{Li}(Y) \ll Y/\log Y$ on the logarithmic integral. 
\end{proof}

\begin{corollary}
\label{sum of f(y)}
    Let $\alpha_{\chi}$ be as in Lemma \ref{sum of f(p)}. Then for all sufficiently small $\epsilon > 0$
    \[
    \sum_{y \leq Y}f(y) = c_{\chi}Y(\log Y)^{\alpha_{\chi} - 1} + O_{\eps}(Y(\log Y)^{-1+\epsilon}), 
    \]
    where 
    \[
    c_{\chi} = \f{1}{\Gamma(\alpha_{\chi})}\prod_{p}\l(\f{1}{1-p^{-1}f(p)}\r)\l(1-\f{1}{p}\r)^{\alpha_{\chi}}.
    \]
\end{corollary}

\begin{proof}
    The result follows from Lemma \ref{sum of f(p)} and an application of the LSD method \cite{LSDmethod}*{Theorem 1.1}. Specifically, Corollary \ref{sum of f(p)log(p)} verifies the hypothesis \cite{LSDmethod}*{Equation (1.2)} holds, and in the notation of \cite{LSDmethod}*{Theorem 1.1}, we choose $A = 1$, $J = 0$ and $k = 0$.
\end{proof}

We now proceed with the estimation of $\sum_{\ba \in \RBw} f(\ba)$. For the remainder of this section, we let $\delta>0$ be a sufficiently small positive parameter depending only on $n$, which for convenience may take different values at different points in the argument. Let $f_{i}(a_i) = \chi_{i}(a_i)\widetilde{g}(a_i)$ so that $f(\ba) = \prod_{i=1}^3 f_{i}(a_i)$. Let $w_i' = w_i\gcd(b_i,M)$ and $Y_{i,\bw} = B_i/w_i'$. Then 
\begin{align*}
    \sum_{\ba \in \mc{R}_{\bm{B},\bw}}f(\ba) &= \prod_{i=1}^3 \sum_{\ss{a_i \in \N\\ |a_i| \leq Y_{i,\bw}\, \forall{i}}}f_i(a_i)\\
    &=\prod_{i=1}^3 c_{\chi_i}Y_{i,\bw}(\log Y_{i,\bw})^{\alpha_{\chi_i}-1}\l(1 + O\l(Y_{i,\bw}\r)^{-\delta}\r).
\end{align*}
Applying this estimate in (\ref{moebius on h}), we obtain 
\begin{equation}
\label{partial sum of h}
    \sum_{\ba \in \mc{R}_{\bm{B}}}h(\ba) = \prod_{i=1}^3 \f{c_{\chi_i}B_i(\log B_i)^{\alpha_{\chi_i}-1}}{\gcd(b_i,M)}\sum_{\bw \in \mc{R}_{\bm{B}}}\f{K(\bw)f(\bw)}{w_1w_2w_3}\l(\prod_{i=1}^3\f{\log(B_i/w_i')^{\alpha_{\chi_i-1}}}{(\log B_i)^{\alpha_{\chi_i}-1}}\r).
\end{equation}
We extract the main term by restricting the $\bw$-sum in (\ref{partial sum of h}) to the smaller region $\bw \in \mc{R}_{\log(\bm{B})}$, where $\log(\bm{B}) = (\log(B_1), \log(B_2), \log(B_3))$. In this region, we have 
\[
\l(\prod_{i=1}^3\f{\log(B_i/w_i')^{\alpha_{\chi_i-1}}}{(\log B_i)^{\alpha_{\chi_i}-1}}\r) = 1 + O\l(\prod_{i=1}^3\l(\f{\log \log B_i}{\log B_i}\r)^{\alpha_{\chi_i}-1}\r) = 1 + O(E_{\delta}(\bm{B})),
\]
where $E_{\delta}(\bm{B}) = (\log (\min_{1\leq i \leq 3}(B_i)))^{-\delta}$. Consequently, 
\[
\sum_{\bw \in \mc{R}_{\log(\bm{B})}}\f{K(\bw)f(\bw)}{w_1w_2w_3}\l(\prod_{i=1}^3\f{\log(B_i/w_i')^{\alpha_{\chi_i-1}}}{(\log B_i)^{\alpha_{\chi_i}-1}}\r) = \sum_{\bw \in \mc{R}_{\log(\bm{B})}}\f{K(\bw)f(\bw)}{w_1w_2w_3}(1+O(E_{\delta}(\bm{B}))).
\]
When $\bw \in \mc{R}_{\bm{B}}\backslash \mc{R}_{\log(\bm{B})}$, we use the trivial bound $\log(B_i/w_i')/\log(B_i) \leq 1$ in (\ref{partial sum of h}). Therefore, 
\[
\begin{split}
\sum_{\ba \in \mc{R}_{\bm{B}}}h(\ba) = &\l(\prod_{i=1}^3 \f{c_{\chi_i}B_i(\log B_i)^{\alpha_{\chi_i}-1}}{\gcd(b_i,M)}\r)(1+O(E_{\delta}(\bm{B})))\\
&\times \l(\sum_{\bw \in \N^3}\f{K(\bw)f(\bw)}{w_1w_2w_3} + O\l(\sum_{\ss{\bw \in \N^3\\ w_i > \log(B_i) \textrm{ for some }i }}\f{K(\bw)f(\bw)}{w_1w_2w_3}\r)\r).
\end{split}
\]
Recalling the definition of $G(\bs)$ and its absolute convergence in the region $\mc{H}_{1/2}$ proved in Lemma \ref{lem: analytic function G}, we conclude that 
\[
\sum_{\ba \in \mc{R}_{\bm{B}}}h(\ba) = G(1,1,1)\prod_{i=1}^3 \f{c_{\chi_i}B_i(\log B_i)^{\alpha_{\chi_i}-1}}{\gcd(b_i,M)}(1+O(E_{\delta}(\bm{B}))).
\]
Below we write $G=G_{\bm{\chi}}$ to make the dependence of $G$ on $\bm{\chi}$ clear, and let $\bm{\chi}_{0} = (\chi_{1, 0}, \chi_{2, 0},\chi_{3, 0})$, where $\chi_{i, 0}$ is the principal character modulo $M_i$. 

Among Dirichlet characters modulo $M_i$, the value of $\alpha_{\chi_i}$ is largest when $\chi_i = \chi_{i, 0}$. Combined with the fact that $G_{\bm{\chi}_{0}}(1,1,1) >0$, it follows that the contributions from non-principal $\chi_i$ produce a smaller power of $\log B_i$ and can be placed into the error term. Moreover, observe that $\alpha_{\chi_{i,0}}$, and hence also $c_{\chi_{i, 0}}$, do not depend on $i$; consequently we may denote them by $\alpha$ and $c$ respectively. Going back to (\ref{nicer form for NfMain}), we have $\bm{\chi}_{0}(\bm{\eps}) = \bm{\chi}_{i,0}(b_i') = 1$, and hence 
\[
N_{\textup{f}}^{\text{main}}(\mathbf{B}; \mathbf{s}, \mathbf{b}, M) = C_n B_1B_2B_3(\log B_1\log B_2\log B_3)^{\alpha - 1}(1 +E_{\delta}(\bm{B})),
\]
where
\begin{equation}\label{def: alpha}
\alpha = \f{1}{\phi(M)}\sum_{r \in (\Z/M\Z)^{\times}}\f{1}{\gcd(n,r-1)} =  \f{1}{\phi(n)}\sum_{r \in (\Z/n\Z)^{\times}}\f{1}{\gcd(n,r-1)}
\end{equation}
and 
\[
 C_n = \f{c^3G_{\bm{\chi}_{0}}(1,1,1)\#\{\{\pm 1\}^3 \cap \mc{D}(n)\}}{\prod_{i=1}^3\phi(M_i)\gcd(b_i,M)}.
\]
Note that $\#\{\{\pm 1\}^3 \cap \mc{D}(n)\}$ is equal to $8$ if $n$ is odd and $6$ if $n$ is even. 
\subsection{Euler product expansion}

We have 
\[
c^3 = \f{1}{\Gamma(\alpha)^3}\prod_p \l(\f{1}{1-p^{-1}f(p)}\r)^3\l(1-\f{1}{p}\r)^{3\alpha}
\]
and 
\[
G_{\bm{\chi}_{0}}(1,1,1) = H(1,1,1)F(1,1,1)^{-1} = \prod_p\l(\sum_{\br \in (\ZZ_{\geq 0})^3}\f{h(\bm{p}^{\br})}{\bm{p}^{\br}}\r)\prod_p\l(\sum_{\br \in (\ZZ_{\geq 0})^3}\f{f(\bm{p}^{\br})}{\bm{p}^{\br}}\r)^{-1}.
\]
Observing that 
\[
\l(\f{1}{1-p^{-1}f(p)}\r)^3 = \l(\sum_{r \geq 0}\f{f(p^r)}{p^r}\r)^3 = \sum_{\br \in (\ZZ_{\geq 0})^3} \f{f(\bm{p}^{\br})}{\bm{p}^{\br}},
\]
we see that
\[
c^3G_{\bm{\chi}_0}(1,1,1) = \f{1}{\Gamma(\alpha)^3}\prod_p \l(1-\f{1}{p}\r)^{3\alpha}\prod_p\l(\sum_{\br \in (\ZZ_{\geq 0})^3} \f{h(\bm{p}^{\br})}{\bm{p}^{\br}}\r).
\]
Suppose that $p \in S(n)$. Then $h(\bm{p}^{\br}) = \bm{\chi}(\bm{p}^{\br})$. Since $\mathbf{b}$ is admissible, we have the equality $\gcd(b_1,b_2,b_3,M) = 1$, which implies that $p\mid M_i$ for some $i$. Therefore $\bm{\chi}_0(\bm{p}^{\br}) = 0$ unless $\br = (0,0,0)$. Recalling the assumption that $(\mathbf{s}, \mathbf{b})$ is admissible, it follows that
\[
\sum_{\br \in (\ZZ_{\geq 0})^3} \f{h(\bm{p}^{\br})}{\bm{p}^{\br}} = 1.
\]
It remains to consider the case $p \notin S(n)$. Here, for admissible $(\mathbf{s}, \mathbf{b})$ we always have $\bm{\chi}_0(\bm{p}^{\br}) = 1$, and so $h(\bm{p}^{\br})$ is either equal to $0$ or $g(p)$, and equals $0$ unless $\br \in \mc{B}$. Therefore, 
\[
\prod_p\l(\sum_{\br \in (\ZZ_{\geq 0})^3} \f{h(\bm{p}^{\br})}{\bm{p}^{\br}}\r) = 1+\sum_{t\geq 1}\f{A(p,t)g(p)}{p^t}, 
\]
where
\[
A(p,t) = \#\{(r_1, r_2, r_3) \in \mc{B}: r_1+r_2+r_3 = t, \gcd(s_1p^{r_1},s_2p^{r_2},s_3p^{r_3}) = 1\}.
\]
Since the set $\mc{B}$ consists of elements of the form $(a,0,0), (a,a,0)$ for $1\leq a <n$ up to permutation of the coordinates, we have the following cases, where $\bm{1}$ denotes the indicator function:
\begin{enumerate}[label=(\roman*)]
    \item If $p\nmid s_1s_2s_3$, then $A(p,t) = 3(\bm{1}_{1\leq t <n}(t) + \bm{1}_{\ss{1\leq t <2n\\ t \textrm{ even }}}(t))$.
    \item If $p$ divides exactly one of $s_1,s_2,s_3$, then $A(p,t) = 3(\bm{1}_{1\leq t <n}(t)) + 2(\bm{1}_{\ss{1\leq t <2n\\ t \textrm{ even }}}(t))$.
    \item If $p$ divides exactly two of $s_1,s_2,s_3$, then $A(p,t) = 2(\bm{1}_{1\leq t <n}(t)) + \bm{1}_{\ss{1\leq t <2n\\ t \textrm{ even }}}(t)$.
    \item If $p$ divides all of $s_1,s_2,s_3$, then $A(p,t) = 0$.
\end{enumerate}
We recall that $g(p) = \gcd(n, p-1)^{-1}$ for any $p \notin S(n)$. Therefore, for any prime $p \nmid s_1s_2s_3$, we have 
\begin{align*}
\sum_{\br \in (\ZZ_{\geq 0 })^3} \f{h(\bm{p}^{\br})}{\bm{p}^{\br}} &= 1+\f{3}{\gcd(n,p-1)}\l(\sum_{1\leq t <n} p^{-t} + \sum_{1\leq t <n} p^{-2t}\r) \\
&= 1+\f{3}{\gcd(n, p-1)}\l(\f{p^{-1}-p^{-n}}{1-p^{-1}} + \f{p^{-2}-p^{-2n}}{1-p^{-2}}\r),
\end{align*}
and similarly for the other cases. Hence the Euler factor for primes $p\notin S(n), p\nmid s_1s_2s_3$ is 
\[
\l(1-\f{1}{p}\r)^{3\alpha}\l(1 + \f{3}{\gcd(n, p-1)}\l(\f{p^{-1}-p^{-n}}{1-p^{-1}} + \f{p^{-2}-p^{-2n}}{1-p^{-2}}\r)\r).
\]
We summarise with the following theorem.

\begin{theorem}
\label{thm: Main term wrap-up}
For any admissible pair $(\mathbf{s}, \mathbf{b})$, there exists $\delta >0$, depending only on $n$, and a constant $C_n > 0$ such that 
\[
N_{\textup{f}}^{\textup{main}}(\mathbf{B}; \mathbf{s}, \mathbf{b}, M) = C_n B_1B_2B_3(\log B_1\log B_2\log B_3)^{\alpha - 1}(1 +E_{\delta}(\bm{B})),
\]
where $\alpha$ is given in (\ref{def: alpha}) and $E_{\delta}(\bm{B}) \ll \log(\min(B_i))^{-\delta}$. The constant $C_n$ is given explicitly by the formula 
\[
C_n = \f{\delta_{\infty}}{\Gamma(\alpha)^3\prod_{i = 1}^3\phi(M_i)\gcd(b_i,M)}  \prod_p C_{n, p},
\]
where $M_i = M/\gcd(b_i,M)$, 
\begin{equation}\label{def delta infinity}
\delta_{\infty}= \begin{cases}
    8, &\textup{ if }n\textup{ is odd},\\
6, &\textup{ if }n\textup{ is even},
\end{cases}
\end{equation}
and 
\[
C_{n, p} := \l(1-\f{1}{p}\r)^{3 \alpha} \l(1 + \f{1}{\gcd(n,p-1)} \l(\f{j_1(p^{-1}-p^{-n})}{1-p^{-1}} + \f{j_2(p^{-2}-p^{-2n})}{1-p^{-2}}\r)\r),
\]
with
\[
(j_1,j_2) := 
\begin{cases}
(0,0), &\textup{ if }p \in S(n) \textup{ or }p \textup{ divides all of }s_1,s_2,s_3,\\
(2,1), &\textup{ if }p \notin S(n) \textup{ and }p \textup{ divides exactly two of }s_1,s_2,s_3,\\
(3,2), &\textup{ if }p \notin S(n) \textup{ and }p \textup{ divides exactly one of }s_1,s_2,s_3,\\
(3,3), &\textup{ if }p \notin S(n) \textup{ and }p \nmid s_1s_2s_3.    
\end{cases}
\]
\end{theorem}

\subsection{Returning to the original sum}
\label{ssRemoveCubefree}
Let $\mathbf{B} = (B, B, B)$. We have the identities
\begin{align*}
N(\mathbf{B}) &= \sum_{\substack{\mathbf{s} \in \N^3 \\ \gcd(s_1, s_2, s_3) = 1}} \hspace{-0.3cm} N_{\text{f}}\left(\left(\frac{B}{s_1^n}, \frac{B}{s_2^n}, \frac{B}{s_3^n}\right); \mathbf{s}\right) 
&= \sum_{\substack{\mathbf{s} \in \N^3}} \sum_{\ss{\mathbf{b} \bmod M \\
(\mathbf{s}, \mathbf{b})\text{ admissible}}} \hspace{-0.3cm} N_{\text{f}}\left(\left(\frac{B}{s_1^n}, \frac{B}{s_2^n}, \frac{B}{s_3^n}\right); \mathbf{s}, \mathbf{b}, M\right).
\end{align*}
Using the trivial bound 
$$
N_{\text{f}}((B_1, B_2, B_3); \mathbf{s}, \mathbf{b}, M) \ll B_1B_2B_3,
$$
we may restrict $s_1, s_2, s_3$ to $(\log B)^C$ for some constant $C > 0$ depending only on $n$. We will now apply Theorem \ref{thm: Main term wrap-up}. Since $s_1, s_2, s_3$ are no longer fixed for us, we shall write $C_{\bs, \bb, n}$ for the leading constant $C_n$ from Theorem \ref{thm: Main term wrap-up} to keep track of the dependence on $\bs$ and $\bb$. Then we get
\begin{multline*}
\sum_{\substack{\mathbf{s} \in \N^3}} \sum_{\ss{\mathbf{b} \bmod M \\
(\mathbf{s}, \mathbf{b})\text{ admissible}}} N_{\text{f}}\left(\left(\frac{B}{s_1^n}, \frac{B}{s_2^n}, \frac{B}{s_3^n}\right); \mathbf{s}, \mathbf{b}, M\right) = \\ \sum_{\substack{0 < s_1, s_2, s_3 \leq (\log B)^C }} \sum_{\ss{\mathbf{b} \bmod M \\
(\bs, \bb)\text{ admissible}}} \frac{C_{\bs, \bb, n}  B^3 (\log B)^{3\alpha - 3}}{s_1^n s_2^n s_3^n} (1 + O((\log B)^{-\delta})).
\end{multline*}
After having done this, we may complete the sum over $s_1, s_2, s_3$ to infinity again with an acceptable error. It thus remains to calculate
$$
\sum_{\substack{\bs \in \N^3 }} \frac{1}{s_1^n s_2^n s_3^n} \sum_{\mathbf{b} \bmod M \text{ admissible}} C_{\bs, \bb, n}.
$$
Writing $C_{\bs, n, p}$ for the factors in the Euler product from Theorem \ref{thm: Main term wrap-up}, we may factorise the constant $C_{\bs, \bb, n}$ into its corresponding pieces
$$
\frac{\delta_\infty}{\Gamma(\alpha)^3} \sum_{\mathbf{s} \in \N^3} \frac{\prod_p C_{\bs, n, p}}{s_1^n s_2^n s_3^n} \sum_{\ss{\mathbf{b} \bmod M \\
(\bs, \bb)\text{ admissible}}} \l(\f{1}{\prod_{i = 1}^3\phi(M_i)\gcd(b_i,M)} \r),
$$
where $\delta_\infty$ is given in (\ref{def delta infinity}). For all $p\in S(n)$, we have $v_p(M) \geq n$. Since admissible $\bb$ must be equivalent modulo $M$ to some $\ba$ with each $a_i$ being $n$-powerfree, this implies that $v_p(b_i) < n$ for all $p \in S(n)$ and all $i$. Using multiplicativity of $\phi$ together with the fact that $\phi(p^r) = p^{r-1}\phi(p)$ for all $r\geq 1$, we deduce that 
\[
\l(\f{1}{\prod_{i = 1}^3\phi(M_i)\gcd(b_i,M)} \r) = \f{1}{\phi(M)^3}.
\]
Thus we get
$$
\frac{\delta_{\infty}}{\Gamma(\alpha)^3 \phi(M)^3} \sum_{\substack{\bs \in \N^3}} \frac{\prod_p C_{\bs, n, p}}{s_1^n s_2^n s_3^n} \sum_{\substack{\mathbf{b} \bmod M\\ (\bs,\bb) \textrm{ admissible}}} 1.
$$
By multiplicativity we may split the sum over $s_1, s_2, s_3$ into pieces, so that the original sum
$$
\sum_{\substack{\bs \in \N^3}} \frac{\prod_p C_{\bs, n, p}}{s_1^n s_2^n s_3^n} \sum_{\substack{\mathbf{b} \bmod M\\ (\bs,\bb) \textrm{ admissible}}} 1
$$
equals
\begin{align}
\label{eGoodBadSplit}
\left(\sum_{\substack{\bs \in \N^3 \\ \gcd(s_1, s_2, s_3) = 1 \\ p \mid s_1s_2s_3 \Rightarrow p \not \in S(n)}} \frac{\prod_p C_{\bs, n, p}}{s_1^n s_2^n s_3^n}\right) \times \left(\sum_{\substack{\mathbf{b} \bmod M \text{ admissible} }} \sum_{\substack{\bs \in \N^3 \\ \gcd(s_1b_1, s_2b_2, s_3b_3,M) = 1 \\ p \mid s_1s_2s_3 \Rightarrow p \in S(n)}} \frac{1}{s_1^n s_2^n s_3^n}\right).
\end{align}
We will evaluate the two sums individually, starting with the first sum. 

\subsubsection{The Euler product for large primes}
We will now treat the first sum in equation \eqref{eGoodBadSplit}, i.e.
\begin{align}
\label{eGoodSplit}
\sum_{\substack{\bs \in \N^3 \\ \gcd(s_1, s_2, s_3) = 1 \\ p \mid s_1s_2s_3 \Rightarrow p \not \in S(n)}} \frac{\prod_p C_{\bs, n, p}}{s_1^n s_2^n s_3^n}.
\end{align}
Defining
\begin{align*}
E^{\text{univ}}_{n, p} &:= 1 + \f{1}{\gcd(n,p-1)} \l(\f{3(p^{-1}-p^{-n})}{1-p^{-1}} + \f{3(p^{-2}-p^{-2n})}{1-p^{-2}}\r), \\
C^{\text{univ}}_n &:= \prod_p E^{\text{univ}}_{n, p} \l(1-\f{1}{p}\r)^{3 \alpha}, \\
E'_{\bs, n, p} &:= 1 + \f{1}{\gcd(n,p-1)} \l(\f{j_1(p^{-1}-p^{-n})}{1-p^{-1}} + \f{j_2(p^{-2}-p^{-2n})}{1-p^{-2}}\r),
\end{align*}
where $j_1, j_2$ depend on $\bs$ as in Theorem \ref{thm: Main term wrap-up}, we obtain
$$
\sum_{\substack{\bs \in \N^3 \\ \gcd(s_1, s_2, s_3) = 1 \\ p \mid s_1s_2s_3 \Rightarrow p \not \in S(n)}} \frac{\prod_p C_{\bs, n, p}}{s_1^n s_2^n s_3^n} = C^{\text{univ}}_n \sum_{\substack{\bs \in \N^3 \\ \gcd(s_1, s_2, s_3) = 1 \\ p \mid s_1s_2s_3 \Rightarrow p \not \in S(n)}} \frac{1}{s_1^n s_2^n s_3^n} \prod_{\substack{p \not \in S(n) \\ p \mid s_1s_2s_3}} \frac{E'_{\mathbf{s}, n, p}}{E^{\text{univ}}_{n, p}}.
$$
Then we calculate
$$
\sum_{\substack{\bs \in \N^3 \\ \gcd(s_1, s_2, s_3) = 1 \\ p \mid s_1s_2s_3 \Rightarrow p \not \in S(n)}} \frac{1}{s_1^n s_2^n s_3^n} \prod_{\substack{p \not \in S(n) \\ p \mid s_1s_2s_3}} \frac{E'_{\mathbf{s}, n, p}}{E^{\text{univ}}_{n, p}} = \prod_{p \not \in S(n)} \left(\sum_{\substack{(i, j, k) \in \Z_{\geq 0} \\ \min(i, j, k) = 0}} \frac{E'_{(p^{i}, p^{j}, p^{k}), n, p}}{E^{\text{univ}}_{n, p} p^{n(i + j + k)}}\right)
$$
by multiplicativity. The sum over all $(i, j, k)$ splits as a disjoint union over three cases
\begin{itemize}
\item[(i)] $(i, j, k) = (0, 0, 0)$,
\item[(ii)] $i > 0$ and $j = k = 0$ (and the permuted versions),
\item[(iii)] $i, j > 0$ and $k = 0$ (and the permuted versions).
\end{itemize}
For each case we finally calculate that the contribution is
\begin{align*}
\begin{cases}
1 &\text{in case (i)}, \\
\frac{3}{p^n - 1} \times \frac{1 + \f{1}{\gcd(n,p-1)} \l(\f{3(p^{-1}-p^{-n})}{1-p^{-1}} + \f{2(p^{-2}-p^{-2n})}{1-p^{-2}}\r)}{E^{\text{univ}}_{n, p}} &\text{in case (ii)}, \\
\frac{3}{(p^n - 1)^2} \times \frac{1 + \f{1}{\gcd(n,p-1)} \l(\f{2(p^{-1}-p^{-n})}{1-p^{-1}} + \f{1(p^{-2}-p^{-2n})}{1-p^{-2}}\r)}{E^{\text{univ}}_{n, p}} &\text{in case (iii)}.
\end{cases}
\end{align*}
To summarise our computations, we introduce the notations
\begin{gather*}
F_1(n, p) := \frac{1 - p^{-3n}}{(1 - p^{-n})^3} \\ 
F_2(n, p) := \f{3 (1 - p^{-2n})(p^{-1}-p^{-n})}{\gcd(n,p-1) (1 - p^{-1}) (1 - p^{-n})^3} \\ 
F_3(n, p) := \f{3 (p^{-2}-p^{-2n})}{\gcd(n,p-1) (1 - p^{-n})^2 (1-p^{-2})}.
\end{gather*}
Adding up the contributions from the three cases, we have shown that equation \eqref{eGoodSplit} equals
$$
\sum_{\substack{\bs \in \N^3 \\ \gcd(s_1, s_2, s_3) = 1 \\ p \mid s_1s_2s_3 \Rightarrow p \not \in S(n)}} \hspace{-0.3cm} \frac{\prod_p C_{\bs, n, p}}{s_1^n s_2^n s_3^n} = \prod_{p \in S(n)} \left(1 - \frac{1}{p}\right)^{3 \alpha} \prod_{p \not \in S(n)} \left(F_1(n, p) + F_2(n, p) + F_3(n, p)\right) \left(1 - \frac{1}{p}\right)^{3 \alpha}.
$$

\subsubsection{The Euler product for small primes}
Recall that the second sum from \eqref{eGoodBadSplit} is
$$
\sum_{\substack{\mathbf{b} \bmod M \text{ admissible} }} \sum_{\substack{\bs \in \N^3 \\ \gcd(s_1b_1, s_2b_2, s_3b_3,M) = 1\\ p \mid s_1s_2s_3 \Rightarrow p \in S(n)}} \frac{1}{s_1^n s_2^n s_3^n}.
$$
Applying the Chinese remainder theorem, we see that this equals
$$
\prod_{p \in S(n)} \left(\sum_{\substack{\mathbf{b} \bmod p^{v_p(M)} \text{ admissible}}} \sum_{\substack{i, j, k \geq 0 \\ \gcd(p^i b_1, p^j b_2, p^k b_3, p^{v_p(M)}) = 1}} \frac{1}{p^{n(i + j + k)}}\right).
$$
Partitioning $\mathbf{b}$ according to the $p$-adic valuations and using the identities \eqref{eGeo1}, \eqref{eGeo2}, \eqref{eGeo3} simplifies the product to
$$
\prod_{p \in S(n)} \left(\frac{1 - p^{-3n}}{(1 - p^{-n})^3} m_1(p) + \frac{3(1 - p^{-2n})}{(1 - p^{-n})^3} m_2(p) + \frac{3}{(1 - p^{-n})^2} m_3(p)\right),
$$
where $m_1(p), m_2(p), m_3(p)$ are 
\begin{equation*}
    \begin{split}
        m_1(p) &:= \# \left\{(a_1, a_2, a_3) \in (\Z_p/ n^2 p^n\Z_p)^3: \begin{array}{c}
             p \nmid a_1,a_2,a_3,   \\
             a_1 t_1^n + a_2 t_2^n + a_3 t_3^n \text{ soluble in } \Z_p/n^2p^n\Z_p
        \end{array}\right\} \\
        m_2(p) &:= \# \left\{(a_1, a_2, a_3) \in (\Z_p/ n^2 p^n\Z_p)^3: \begin{array}{c}
             0 < v_p(a_1) < n, p \nmid a_2,a_3,   \\
             a_1 t_1^n + a_2 t_2^n + a_3 t_3^n \text{ soluble in } \Z_p/n^2p^n\Z_p
        \end{array}\right\}  \\
        m_3(p) &:= \# \left\{(a_1, a_2, a_3) \in (\Z_p/ n^2 p^n\Z_p)^3: \begin{array}{c}
              0 < v_p(a_1) = v_p(a_2) < n, p \nmid a_3,   \\
             a_1 t_1^n + a_2 t_2^n + a_3 t_3^n \text{ soluble in } \Z_p/n^2p^n\Z_p
        \end{array}\right\}.
    \end{split}
\end{equation*}
We will revisit the quantities $m_1(p), m_2(p), m_3(p)$ once more in Lemma \ref{lem:local_densities}.

\subsection{Conclusion}
We summarise our work in the following theorem. 

\begin{theorem}\label{thm:main summary}
We have
$$
N(\mathbf{B}) \sim \frac{(1 + \mathbf{1}_{n \equiv 0 \bmod 2}) \delta_{\infty} B^3 (\log B)^{3\alpha - 3}}{\Gamma(\alpha)^3 \phi(M)^3} \times E_1 \times E_2
$$
with
$$
E_1 := \prod_{p \in S(n)} \left(\frac{1 - p^{-3n}}{(1 - p^{-n})^3} m_1(p) + \frac{3(1 - p^{-2n})}{(1 - p^{-n})^3} m_2(p) + \frac{3}{(1 - p^{-n})^2} m_3(p)\right) \left(1 - \frac{1}{p}\right)^{3 \alpha}
$$
and
$$
E_2 := \prod_{p \not \in S(n)} \left(F_1(n, p) + F_2(n, p) + F_3(n, p)\right) \left(1 - \frac{1}{p}\right)^{3 \alpha}.
$$
\end{theorem}

\noindent This is immediately verified to be equivalent to \Cref{tMain}.

%% file: Section6.tex
\section{Comparison with Loughran--Rome--Sofos}
\label{sLRS}
In this last section we will show that Theorems \ref{tMain} and \ref{thm:leadingconstant} agree with the prediction of Loughran--Rome--Sofos. Let $\mathcal{X}_n \subset \Proj^2_{a_1, a_2, a_3} \times \Proj^2_{t_1, t_2, t_3}$ be the projective variety defined by the universal generalised Fermat equation
\begin{equation}
\label{eq:universal fermat equation}
a_1 t_1^n + a_2 t_2^n + a_3 t_3^n = 0.
\end{equation}
The Jacobian criterion shows that this variety is smooth and thus geometrically integral.  Let $f_n: \mathcal{X}_n \to \Proj^2_{a_1, a_2, a_3}$ be the projection on the first factor. The generic fiber of $f_n$ is the geometrically integral plane curve over $\Q(a_1/a_3, a_2/a_3)$ given by dividing \eqref{eq:universal fermat equation} by $a_3$. The assumptions at the start of \cite{LRS}*{\S3} are thus satisfied for $f_n$, and so \cite{LRS}*{Conjecture 3.8, see also \S3.7.1} makes the following prediction.

\begin{conjecture}[]\label{conj:LRS for our setting}
Let $n > 1$ be an integer, and write $X:=\PP^2_{a_1,a_2,a_3}$. Then as $B\to\infty$
\[
\# \set{[a_1:a_2:a_3]\in\PP^2_\QQ : \substack{H([a_1:a_2:a_3]) \leq B, \\ a_1t_1^n+a_2t_2^n+a_3t_3^n=0\textnormal{ has a} \\\textnormal{non-zero solution everywhere-locally}}}\sim \widetilde{C}_n B (\log B)^{-\Delta(f_n)},
\]
where the constant $\widetilde{C}_n > 0$ is given explicitly by
\[\widetilde{C}_n:=\frac{\alpha^*(X)\abs{\Br_{\textnormal{sub}}(X,f_n)/\Br(\QQ)}\tau_{f_n}\braces{\braces{\prod_{v}f_n(\cX_n(\QQ_v))}^{\Br_{\textnormal{sub}}(X,f_n)}}}
{\prod_{D\in X^{(1)}}\Gamma(\delta_D(f_n))}\prod_{D\in X^{(1)}}\eta(D)^{1-\delta_D(f_n)},\]
and the remaining notations are defined in \cite{LRS}.
\end{conjecture}
The rest of this section will be concerned with verifying that our results Theorems \ref{tMain} and \ref{thm:leadingconstant} prove this prediction.

\subsection{Computing terms}
We now determine the various invariants present in \Cref{conj:LRS for our setting}.

\subsubsection{Delta invariants}
The first step is to compute $\delta_D(f_n)$, as defined in \cite{LRS}*{Def.~3.2}, for all divisors $D \subset \Proj^2$. 
\begin{lemma}\label{lem:computation of delta_D}
We have 
\[\delta_D(f_n)=\begin{cases}
    \alpha_n&\textnormal{if }D=\set{a_i=0}\textnormal{ for some }i\\
    1&\textnormal{else}
\end{cases}\]
where 
\[
\alpha_n := \frac{1}{\varphi(n)} \sum_{r \in (\Z / n \Z)^{\times}} \frac{1}{\gcd(n, r - 1)}.
\]
\end{lemma}

\begin{proof}
The Jacobian criterion shows that $f_n$ is smooth over the locus $\{a_1 a_2 a_3 \neq 0\}$ so the only curves $D \subset \Proj^2$ for which $\delta_D(f_n) \neq 1$ are $\{a_1 = 0\}, \{a_2 = 0\}$ and $\{a_3 = 0\}$.  By symmetry it suffices to compute $\delta_{\{a_3 = 0\}}(f_n)$. Let $A = -\frac{a_1}{a_2}$, the fiber $f_n^{-1}(\{a_3 = 0\})$ over the generic point of $\{a_3 = 0\}$ is the $\Q(A)$ variety defined by the equation $A t_1^n = t_2^n$.

Let $\zeta_n \in \overline{\Q}$ be a primitive $n$-th root of unity and $\sqrt[n]{A} \in \overline{\Q(A)}$ an $n$-th root of $A$. Over the Galois extension $\Q(\zeta_n, \sqrt[n]{A})/\Q(A)$ the variety $f_n^{-1}(\{a_3 = 0\})$ decomposes into the $n$ geometrically irreducible components 
\[
D_x := \{ \zeta^x \sqrt[n]{A} t_1 = t_2 \}
\] 
for $x \in \Z / n\Z$. 

Note that $\Gal(\QQ(\zeta_n,\sqrt[n]{A})/\QQ(A))=\set{\gamma_{r,q}:q\in\ZZ/n\ZZ\ r\in\ZZ/n\ZZ^\times}$ where $\gamma_{r,q}$ is the unique automorphism extending the map
\begin{align*}
    \sqrt[n]{A}&\mapsto \zeta_n^q\sqrt[n]{A}\\
    \zeta_n&\mapsto \zeta_n^r.
\end{align*}
Clearly, $\gamma_{r,q}$ sends $D_x$ to $D_{r x + q}$. Unfolding the definition of the $\delta$-invariant we see that
\[
\delta_{\{a_3 = 0\}}(f_n) = \frac{\#\{(q,r) \in \ZZ / n \ZZ \times (\ZZ / n \ZZ)^{\times} : \text{there exists } x \in \ZZ / n \ZZ \text{ such that } rx + q = x \}}{n \varphi(n)}.
\]
The equation $r x + q = x$ is equivalent to $(r - 1) x = -q$ and thus has a solution modulo $n$ if and only if $\gcd(n, r - 1)$ divides $q$. We thus have
\[
\delta_{\{a_3 = 0\}}(f_n) = \frac{1}{n \varphi(n)} \sum_{r \in (\Z / n \Z)^{\times}} \sum_{\substack{q \in \Z / n \Z \\ \gcd(n, r - 1)| q}} 1 = \frac{1}{\varphi(n)} \sum_{r \in (\Z / n \Z)^{\times}} \frac{1}{\gcd(n, r - 1)}=\alpha_n,
\]
as desired.
\end{proof}

An immediate corollary from the definition \cite{LRS}*{Def.~3.2} of $\Delta(f_n)$ is then the following.

\begin{corollary}
\label{cor:computation_of_Delta}
We have
\[
\Delta(f_n) := \sum_{i=1}^3\braces{1-\delta_{\{a_i=0\}}(f_n)} = 3(1 - \alpha_n).
\]
\end{corollary}

\subsubsection{Subordinate Brauer groups and obstructions}
An important invariant of $f_n$ appearing in the conjecture is the \emph{subordinate Brauer group} $\Br_{\text{sub}}(\Proj^2, f_n)$, see \cite{LRS}*{Def.~2.1}. The following lemma computes this group.

\begin{lemma}
\label{lem:subordinate_Brauer_group}
Let $f_n$ be as above.
\begin{itemize}
\item If $n$ is odd, then $\Br_{\emph{sub}}(\Proj^2, f_n) = \Br \Q$. 
\item If $n$ is even, then $\Br_{\emph{sub}}(\Proj^2, f_n)/ \Br \Q \cong \Z/ 2 \Z$ and it is generated by the quaternion algebra $(-a_1/a_3, -a_2/a_3)$.
\end{itemize}
\end{lemma}

\begin{proof}
Let $U := \Proj^2 \setminus \{a_1 a_2 a_3 = 0\}$. The map $\mathcal{X}_n \times_{\Proj^2} U \to U$ has geometrically integral fibers so $\Br_{\text{sub}}(\Proj^2, f_n) \subset \Br U$ by Grothendieck purity, see \cite{Col21}*{Cor.~3.7.3}. Let $b \in \Br U$, by \cite{LRS}*{Lemma 2.5} one has $b \in \Br_{\text{sub}}(\Proj^2, f_n)$ if and only if for all $\{i, j, k\} = \{1, 2, 3\}$ we have
\[
\partial_{\{a_i = 0\}}(b) \in \ker\braces{H^1(\Q(a_j/a_k), \Q/\Z) \to H^1\braces{\Q\braces{\sqrt[n]{-a_j/a_k}}, \Q/\Z}}. 
\]
In other words, the character $\partial_{\{a_i = 0\}}(b) \in H^1(\Q(a_j/a_k), \Q/\Z)$ defines a cyclic extension contained in $\Q(\sqrt[n]{-a_j/a_k})$.

If $n$ is odd, then $\Q(\sqrt[n]{-a_j/a_k})$ contains no cyclic extensions of $\Q(a_j/a_k)$ so $\partial_{\{a_i = 0\}}(b) = 0$ for all $i$. This is equivalent to $b \in \Br \Proj^2$ by Grothendieck purity (\cite{Col21}*{Cor.~3.7.3}) and $\Br \Proj^2 = \Br \Q$ (\cite{Col21}*{Thm.~5.1.3}).

If $n$ is even, then the only cyclic extension contained in $\Q(\sqrt[n]{-a_j/a_k})$ is $\Q(\sqrt[2]{-a_j/a_k})$. We deduce that $\Br_{\text{sub}}(\Proj^2, f_n) = \Br_{\text{sub}}(\Proj^2, f_2)$. The lemma then follows from the computation of $\Br_{\text{sub}}(\Proj^2, f_2)$ in \cite{LRS}*{Lemma 4.11}.
\end{proof}

To compute the leading constant in general one needs to account for the subordinate Brauer--Manin obstruction as in \cite{LRS}*{\S3.7}. But for $f_n$ there is no subordinate Brauer--Manin obstruction.

\begin{corollary}
\label{cor:subordinate_obstruction}
There is no subordinate Brauer--Manin obstruction for $f_n$.  That is,
\[\prod_{v}f_n(\cX_n(\QQ_v))=\braces{\prod_{v}f_n(\cX_n(\QQ_v))}^{\Br_{\textnormal{sub}}(X,f_n)}.\]
\end{corollary}

\begin{proof}
There is nothing to prove when $n$ is odd so we may assume that $n$ is even, in this case let $b := (-a_1/a_3, -a_2/a_3)$ be the generator of $\Br_{\text{sub}}(\Proj^2, f_n)/ \Br \Q$. Let $v$ be a place of $\Q$ and $P_v = [a_{1}:a_{2}:a_{3}] \in f_n(\mathcal{X}_n(\Q_v))$.

The equation $a_1 t_1^n + a_2 t_2^n + a_3 t_3^n = 0$ thus has a solution in $\Q_v$ and since $n$ is even this implies that $a_1 x_1^2 + a_2 x_2^2+ a_3 x_3^2 = 0$ has a solution in $\Q_v$. It follows that 
$$
\text{inv}_v(b(P_v)) = (-a_1/a_3, -a_2/a_3)_v = 0. 
$$
Since the above holds for all $v$, we deduce that $f_n(\mathcal{X}_n(\A_{\Q}))^{\Br_{\text{sub}}(\Proj^2, f_n)} = f_n(\mathcal{X}_n(\A_{\Q}))$.
\end{proof}

\subsubsection{Local Tamagawa measures}
It remains to compute the Tamagawa volume. For each place $v$ let $\tau_v$ be the local Tamagawa measure on $\Proj^2(\Q_v)$ and $\mu_v$ the usual Haar measure on $\Q_v^3$ as in \cite{LRS}*{\S4.1}. Recall the definitions of $\delta_p(n)$ and $\delta_\infty(n)$ in \Cref{thm:leadingconstant}.

\begin{lemma} 
\label{lem:tamagawa_local_densities}
The following facts are true.
    \begin{enumerate} 
        \item $\tau_\infty(f_n(\mathcal{X}_n(\R))) = \frac{3}{2} \delta_\infty(n)$
        \item If $p$ is prime, then $\tau_p(f_n(\mathcal{X}_n(\Q_p))) = (1 - p^{-1})^{-1} \delta_p(n)$.
    \end{enumerate}
\end{lemma}
\begin{proof}
By \cite{LRS}*{Prop. 4.1}, we have
\begin{align*}
\tau_{\infty}(f_n(\mathcal{X}_n(\R))) &= \frac{3}{2} \mu_{\infty}(\{(a_1, a_2, a_3) \in [-1, 1]^3 : a_1 t_1^n + a_2 t_2^n + a_3 t_3^n = 0 \text{ soluble in } \R\}) \\
&= 
\begin{cases}
\frac{3}{2}\mu_{\infty}([-1, 1]^3) = \frac{3}{2} \cdot 8 = 12 &\text{ if $n$ is odd} \\ 
\frac{3}{2}\mu_{\infty}\left([-1, 1]^3 \setminus \left([-1,0)^3 \cup (0, 1]^3 \right)\right) = \frac{3}{2} \cdot 6 = 9 &\text{ if $n$ is even.}
\end{cases}
\end{align*}
If $p$ is a prime, then let
\begin{equation}
\label{eq:locally_soluble_region}
U := \{(a_1, a_2, a_3) \in \Z_p^3 \setminus p \Z_p^3: a_1 t_1^n + a_2 t_2^n + a_3 t_3^n = 0 \text{ soluble in } \Q_p\}.
\end{equation}
By \cite{LRS}*{Prop. 4.1}, we also have $\tau_{p}(f_n(\mathcal{X}_n(\Q_p))) = (1 - p^{-1})^{-1} \mu_p(U)$ and $\mu_p(U) = \delta_p(n)$ by definition.
\end{proof}

Let us now prove Proposition \ref{intro:prop:explicitdensities}.

\begin{lemma}
\label{lem:local_densities}
The following facts hold.
\begin{enumerate}
\item Let $p$ be a prime and define
\begin{equation*}
    \begin{split}
        m_1 &:= \# \left\{(a_1, a_2, a_3) \in (\Z_p/ n^2 p^n\Z_p)^3: \begin{array}{c}
             p \nmid a_1,a_2,a_3,   \\
             a_1 t_1^n + a_2 t_2^n + a_3 t_3^n \emph{ soluble in } \Z_p/n^2p^n\Z_p
        \end{array}\right\} \\
        m_2 &:= \# \left\{(a_1, a_2, a_3) \in (\Z_p/ n^2 p^n\Z_p)^3: \begin{array}{c}
             0 < v_p(a_1) < n, p \nmid a_2,a_3,   \\
             a_1 t_1^n + a_2 t_2^n + a_3 t_3^n \emph{ soluble in } \Z_p/n^2p^n\Z_p
        \end{array}\right\}  \\
        m_3 &:= \# \left\{(a_1, a_2, a_3) \in (\Z_p/ n^2 p^n\Z_p)^3: \begin{array}{c}
              0 < v_p(a_1) = v_p(a_2) < n, p \nmid a_3,   \\
             a_1 t_1^n + a_2 t_2^n + a_3 t_3^n \emph{ soluble in } \Z_p/n^2p^n\Z_p
        \end{array}\right\}.
    \end{split}
\end{equation*}
We then have
\begin{equation}
\label{eq:Local_Tamagawa_measure}
    p^{6v_p(n) + 3n}\delta_p(n) = \frac{1 - p^{-3n}}{(1 - p^{-n})^3}m_1 + \frac{3(1 - p^{-2n})}{(1 - p^{-n})^3}m_2 + \frac{3}{(1 - p^{-n})^2} m_3.
\end{equation}

\item If $p \not \in S(n)$, then $\delta_p(n)$ equals
\begin{multline*}
\frac{(1 - p^{-1})^3(1 - p^{-3n})}{(1 - p^{-n})^3} + \frac{3(1 - p^{-2n})(1 - p^{-1})^2(p^{-1} - p^{-n})}{\gcd(n, p - 1)(1 - p^{-n})^3} + \\
\frac{3(1 - p^{-1})^3(p^{-2} - p^{-2n})}{\gcd(n, p - 1)(1 - p^{-n})^2(1 - p^{-2})}.
\end{multline*}
\end{enumerate}
\end{lemma}

\begin{proof}
Let $p$ be a prime, recall the definition \eqref{eq:locally_soluble_region}  of $U$ and the fact that $\delta_p(n) = \mu_p(U)$. We will split $U$ into the following pieces, where $i \in \{1,2,3\}$
\begin{equation*}
    \begin{split}
    U_{\text{main}} &:=  \{(a_1, a_2, a_3) \in U: v_p(a_1) \equiv v_p(a_2) \equiv v_p(a_3) \equiv 0 \bmod{n}\} \\
    V_i &:= \{(a_1, a_2, a_3) \in U: v_p(a_j) \equiv 0 \bmod{n} \text{ for } j \neq i \text{ and } v_p(a_i) \not \equiv 0 \bmod{n} \} \\
    W_i &:= \{(a_1, a_2, a_3) \in U: v_p(a_j) \not \equiv 0 \bmod{n} \text{ for } j \neq i \text{ and } v_p(a_i) \equiv 0 \bmod{n} \}.
    \end{split}
\end{equation*}
It is clear that $U = U_{\text{main}} \coprod V_1 \coprod V_2 \coprod V_3 \coprod W_1 \coprod W_2 \coprod W_3$. Symmetry considerations give the identities $\mu_p(V_1) = \mu_p(V_2) = \mu_p(V_3)$ and $\mu_p(W_1) = \mu_p(W_2) = \mu_p(W_3)$. It follows that 
$$
\mu_p(U) = \mu_p(U_{\text{main}}) + 3\mu_p(V_1) + 3\mu_p(W_3).
$$
For $a_1, a_2, a_3 \in \Z_p / n^2 p^n\Z_p$ and $k_1, k_2, k_3 \in \Z_{\geq 0}$, we introduce the set
\[
U(a_1, a_2, a_3 ; k_1, k_2, k_3) := \{(b_1, b_2, b_3) \in \Z_p^3 : v_p(b_i) \geq n k_i, b_i/p^{n k_i} \equiv a_i \bmod{n^2 p^n}\}.
\]
It follows from Proposition \ref{prop:Hensel_general} that if $v_p(a_i) < n$ and $(b_1, b_2, b_3) \in U(a_1, a_2, a_3 ; k_1, k_2, k_3)$, then $b_1 t_1^n + b_2 t_2^n + b_3 t_3^n = 0$ is soluble in $\Z_p$ if and only if $a_1 t_1^n + a_2 t_2^n + a_3 t_3^n = 0$ is soluble in $\Z_p/ n^2 p^n \Z_p$. We can thus cover $U_{\text{main}}, V_1, W_3$ as follows
\begin{equation*}
    \begin{split}
        U_{\text{main}} &= \coprod_{\substack{a_1, a_2, a_3 \in \Z_p / n^2 p^n\Z_p, p \nmid a_1, a_2,a_3 \\ a_1 t_1^n + a_2 t_2^n + a_3 t_3^n \text{ soluble in } \Z_p/n^2p^n\Z_p}} \coprod_{\substack{k_1, k_2, k_3  = 0 \\ \min(k_1, k_2, k_3) = 0}}^{\infty} U(a_1, a_2, a_3 ; k_1, k_2, k_3) \\
        V_1 &= \coprod_{\substack{a_1, a_2, a_3 \in \Z_p / n^2 p^n\Z_p, 0 < v_p(a_1) < n, p \nmid a_2,a_3 \\ a_1 t_1^n + a_2 t_2^n + a_3 t_3^n \text{ soluble in } \Z_p/n^2p^n\Z_p}} \coprod_{\substack{k_1, k_2, k_3  = 0 \\ \min(k_2, k_3) = 0}}^{\infty} U(a_1, a_2, a_3 ; k_1, k_2, k_3) \\
        W_3 &= \coprod_{\substack{a_1, a_2, a_3 \in \Z_p / n^2 p^n\Z_p, 0 < v_p(a_1) = v_p(a_2) < n, p \nmid a_3 \\ a_1 t_1^n + a_2 t_2^n + a_3 t_3^n \text{ soluble in } \Z_p/n^2p^n\Z_p}} \coprod_{\substack{k_1, k_2  = 0}}^{\infty} U(a_1, a_2, a_3 ; k_1, k_2, 0)
    \end{split}
\end{equation*}
The set $U(a_1, a_2, a_3 ; k_1, k_2, k_3)$ is a shift of $(p^{n k_1 }, p^{n k_2}, p^{n k_3})n^2 p^n \Z_p^3$ so has Haar measure $p^{- nk_1 - nk_2 - nk_3 - 3n - 6v_p(n)}$. The measure of a disjoint union is the sum of the measures so part 1 follows by evaluating the sum over $k_1,k_2,k_3$ using the following geometric sum identities
\begin{align}
    \label{eGeo1}
    \sum_{\substack{k_1, k_2, k_3  = 0 \\ \min(k_1, k_2, k_3) = 0}}^{\infty} p^{- nk_1 - nk_2 - nk_3} &= (1 - p^{-3n})\sum_{\substack{k_1, k_2, k_3  = 0}}^{\infty} p^{- nk_1 - nk_2 - nk_3} = \frac{1 - p^{-3n}}{(1 - p^{-n})^{3}} \\
    \label{eGeo2}
    \sum_{\substack{k_1, k_2, k_3  = 0 \\ \min(k_2, k_3) = 0}}^{\infty} p^{- nk_1 - nk_2 - nk_3} &= (1 - p^{-2n})\sum_{\substack{k_1, k_2, k_3  = 0}}^{\infty} p^{- nk_1 - nk_2 - nk_3} = \frac{1 - p^{-2n}}{(1 - p^{-n})^{3}} \\
    \label{eGeo3}
    \sum_{k_1, k_2 = 0}^{\infty} p^{- nk_1 - nk_2} &= \frac{1}{(1 - p^{-n})^2}.
\end{align}
If $p \not \in S(n)$, then we can compute $m_1, m_2$ and $m_3$ using Proposition \ref{prop:Hensel_large_primes}.
\begin{enumerate}
\item To compute $m_1$ we note that if $p \nmid a_1, a_2, a_3$, then $a_1 t_1^n + a_2 t_2^n + a_3 t_3^n = 0$ is always soluble in $\Z_p/p^n\Z_p$ (note that $p \not \in S(n)$ implies $n^2 p^n \Z_p = p^n \Z_p$). This implies that $m_1 = p^{3n}(1 - p^{-1})^3$.
\item If $0 < v_p(a_1) < n$ and $a_2, a_3 \in (\Z_p/ p^{n} \Z_p)^{\times}$, then $a_1 t_1^n + a_2 t_2^n + a_3 t_3^n = 0$ is soluble in $\Z_p/p^n\Z_p$ if and only if $a_2/a_3$ is an $n$-th power modulo $p$. There are thus $p^{n-1} - 1$ choices for $a_1$ and for every $a_2$ there are $\frac{p^n - p^{n-1}}{\gcd(n, p-1)}$ possibilities for $a_3$. This shows that $m_2 = \frac{p^{3n}}{\gcd(n, p - 1)}(1 - p^{-1})^2(p^{-1} - p^{-n})$.
\item In the last case $0 < v_p(a_1) = v_p(a_2) = \ell < n$ and $p \nmid a_3$. This is soluble in $\Z_p/p^n\Z_p$ if and only if $a_1/a_2$ is an $n$-th power modulo $p$. Analogous arguments to the previous case show that for any fixed $\ell$ the number of such triples $(a_1, a_2, a_3)$ is $\frac{p^{3n - 2\ell}}{\gcd(n, p - 1)}(1 - p^{-1})^3$. Summing over the possible $\ell$ shows that $m_3 = \frac{p^{3n}}{\gcd(n, p - 1)}(1 - p^{-1})^3(p^{-2} - p^{-2n})(1 - p^{-2})^{-1}$.
\end{enumerate}
Substituting these values into \eqref{eq:Local_Tamagawa_measure} shows that $\delta_p(n)$ is equal to 
\[
\frac{(1 - p^{-1})^3(1 - p^{-3n})}{(1 - p^{-n})^3} + \frac{3(1 - p^{-2n})(1 - p^{-1})^2(p^{-1} - p^{-n})}{\gcd(n, p - 1)(1 - p^{-n})^3} + \frac{3(1 - p^{-1})^3(p^{-2} - p^{-2n})}{\gcd(n, p - 1)(1 - p^{-n})^2(1 - p^{-2})},
\]
as desired.
\end{proof}

\subsection{\texorpdfstring{Verification of \Cref{conj:LRS for our setting}}{Verification of conjecture}}
We have now computed all of the invariants needed for \Cref{conj:LRS for our setting}.

\begin{theorem}
\label{thm:we prove LRS}
\Cref{conj:LRS for our setting} is true.
\end{theorem}

\begin{proof} 
Let us first note that the conjecture is stated for points $[a_1: a_2: a_3] \in \Proj^2(\Q)$, which are in a $1$ to $2$ relation with coprime triples $(a_1, a_2, a_3) \in \Z^3$.  Moreover, the conjecture is stated in terms of the anticanonical height, which is $\max(|a_1|, |a_2|, |a_3|)^3$, and not the standard height. The left-hand side in \Cref{conj:LRS for our setting} is thus $\frac{1}{2} N(B^{\frac{1}{3}}, B^{\frac{1}{3}}, B^{\frac{1}{3}})$.   Moreover, by \Cref{cor:computation_of_Delta}, the conjectural logarithmic power is $-\Delta(f_n)=-3(1-\alpha_n)$ which agrees with \Cref{tMain}.  Hence we need only check the conjectural leading constant $\widetilde{C}_n$ against \Cref{thm:leadingconstant}.

Applying Lemmas \ref{lem:computation of delta_D}, \ref{lem:subordinate_Brauer_group}, \ref{cor:subordinate_obstruction}, and \ref{lem:tamagawa_local_densities}, we obtain
\[\widetilde{C}_n=\frac{\alpha^*(\PP^2)\braces{1+\mathbf{1}_{\textnormal{even}}(n)}\tau_{f_n}\braces{\prod_{v}f_n(\cX_n(\QQ_v))}}
{\Gamma(\alpha_n)^3}\prod_{i=1}^3\eta\braces{\set{a_i=0}}^{1-\alpha_n}.\]
Note that, by \cite{LRS}*{Equation (3.2) or proof of Lemma 4.5}, we have $\alpha^*(\PP^2)=\alpha(\PP^2)=1/3$, where $\alpha$ is Peyre's constant.  Moreover, by \cite{LRS}*{Lemma 3.5} and the remark following it, for each $i$ we have $\eta(\set{a_i=0})=\eta(\mathcal{O}_{\PP^2}(1))=3$.  Hence we have
\[\widetilde{C}_n=\frac{\braces{1+\mathbf{1}_{\textnormal{even}}(n)}3^{2-3\alpha_n}}
{\Gamma(\alpha_n)^3}\tau_{f_n}\braces{\prod_{v}f_n(\cX_n(\QQ_v))}.\]

It remains to compute the Tamagawa measure.  As in \cite{LRS}*{Equation (3.6)} we express this as $\tau_{f_n}=\prod_v\tau_v(f_n(\cX(\QQ_v)))\lambda_v^{-1}$, where the $\lambda_v$ are convergence factors used to normalise the Tamagawa measure.  By \cite{LRS}*{Equation (3.5) and subsequent equation}, since $\set{a_i=0}$ is geometrically irreducible, these are given by $\lambda_\infty=1$ and for primes $p$ we get $\lambda_p=\braces{1-\frac{1}{p}}^{2-3\alpha_n}$.
We then apply \Cref{lem:tamagawa_local_densities} to conclude that
\[\tau_{f_n}\braces{\prod_vf_v\braces{\cX_n(\QQ_v)}}=\frac{3}{2}\delta_\infty(n)\prod_{p}\delta_p(n)\braces{1-\frac{1}{p}}^{3\alpha_n-3}.\]
Hence we have determined that
\[\widetilde{C}_n=\frac{\braces{1+\mathbf{1}_{\textnormal{even}}(n)}\delta_\infty(n)3^{3(1-\alpha_n)}}
{2\Gamma(\alpha_n)^3}\prod_{p}\delta_p(n)\braces{1-\frac{1}{p}}^{3\alpha_n-3}=\frac{3^{3(1-\alpha_n)}}{2}C_n,\]
where $C_n$ is the leading constant of \Cref{thm:leadingconstant}.  Hence the conjecture predicts 
$$
\tfrac{1}{2}N(B^{1/3},B^{1/3},B^{1/3})\sim \frac{3^{3(1-\alpha_n)}}{2}C_n B\log(B)^{-3(1-\alpha_n)}.  
$$
Multiplying by $2$ and performing the change of variables $B\mapsto B^3$ we then obtain the prediction $N(B,B,B)\sim C_n B^3\log(B)^{3\alpha_n-3}$, which is true by \Cref{tMain}.
\end{proof}